\documentclass[12pt]{amsart}

\usepackage[usenames,dvipsnames]{xcolor}

\usepackage{enumitem,kantlipsum}

\usepackage{mathtools}
\usepackage{xfrac}
\usepackage[utf8]{inputenc}
\usepackage{amssymb}
\usepackage{amsmath, amsthm}
\usepackage{amsfonts}
\usepackage{tikz}
\usepackage{algorithmic}
\usepackage{asymptote}
\usepackage{url}
\newtheorem{theorem}{Theorem}[section]
\newtheorem{corollary}{Corollary}[theorem]
\newtheorem{lemma}[theorem]{Lemma}
\newtheorem{proposition}[theorem]{Proposition}

\newtheorem{definition}[theorem]{Definition}

\usepackage[OT2,T1]{fontenc}
\DeclareSymbolFont{cyrletters}{OT2}{wncyr}{m}{n}
\DeclareMathSymbol{\Sha}{\mathalpha}{cyrletters}{"58}

\theoremstyle{definition}

\newcommand{\Pf}{\mathcal{P}_{\text{fin}}}
\newcommand{\X}{\mathcal{X}}
\newcommand{\Y}{\mathcal{Y}}
\newcommand{\N}{\mathbb{N}}
\newcommand{\Z}{\mathbb{Z}}
\newcommand{\Zbar}{\overline{\mathbb{Z}}}

\newcommand{\E}{\mathbb{E}}
\newcommand{\PPP}{\mathbb{P}}

\newcommand{\OO}{\mathcal{O}}
\newcommand{\p}{\mathfrak{p}}
\newcommand{\q}{\mathfrak{q}}
\newcommand{\PP}{\mathcal{P}}

\newcommand{\LL}{\mathcal{L}}
\newcommand{\Q}{\mathbb{Q}}

\newcommand{\Spec}{\text{Spec}}
\newcommand{\disc}{\text{Disc}}

\newcommand{\A}{\mathcal{A}}
\newcommand{\FF}{\mathcal{F}}

\newcommand{\CC}{\mathcal{C}}

\newcommand{\Cl}{\text{Cl}}
\newcommand{\Op}{O\Big(\frac{1}{p}\Big)}
\newcommand\restr[2]{{
  \left.\kern-\nulldelimiterspace 
  #1 
  \vphantom{\big|} 
  \right|_{#2} 
  }}

\newcommand{\legendre}[2]{\genfrac{(}{)}{}{}{#1}{#2}}

\newtheorem{innercustomthm}{Theorem}
\newenvironment{customthm}[1]
  {\renewcommand\theinnercustomthm{#1}\innercustomthm}
  {\endinnercustomthm}

\newtheorem{innercustomprop}{Proposition}
\newenvironment{customprop}[1]
  {\renewcommand\theinnercustomprop{#1}\innercustomprop}
  {\endinnercustomprop}

\begin{document}
\title{Denisty questions in rings of the form $\OO_K[\gamma]\cap K$.}
\author{Deepesh Singhal, Yuxin Lin}

\begin{abstract}
We fix a number field $K$ and study statistical properties of the ring $\OO_K[\gamma]\cap K$ as $\gamma$ varies over algebraic numbers of a fixed degree $n\geq 2$. Given $k\geq 1$, we explicitly compute the density of $\gamma$ for which $\OO_K[\gamma]\cap K =\OO_K[1/k]$ and show that this does not depend on the number field $K$. In particular, we show that the density of $\gamma$ for which $\OO_K[\gamma]\cap K=\OO_K$ is $\frac{\zeta(n+1)}{\zeta(n)}$.
In a recent paper \cite{Paper 1}, the authors define $X(K,\gamma)$ to be a certain finite subset of $\Spec(\OO_K)$ and show that $X(K,\gamma)$ determines the ring $\OO_K[\gamma]\cap K$. We show that if $\p_1,\p_2\in \Spec(\OO_K)$ satisfy $\p_1\cap \Z\neq\p_2\cap \Z$, then the events $\p_1\in X(K,\gamma)$ and $\p_2\in X(K,\gamma)$ are independent. As $t\to\infty$, we study the asymptotics of the density of $\gamma$ for which $|X(K,\gamma)|=t$.
\end{abstract}

\maketitle

\section{Introduction}

Let $\overline{\Q}$ be the algebraic closure of $\Q$ and $\overline{\Z}$ be the ring of algebraic integers.
For $\gamma\in\overline{\Q}$, we denote the minimal polynomial of $\gamma$ in $\Z[x]$ as $F_{\gamma}(x)=a_nx^n+a_{n-1}x^{n-1}+\dots+a_0$. So $F_{\gamma}(x)$ is irreducible in $\Z[x]$, $F_{\gamma}(\gamma)=0$, $a_n>0$ and $\gcd(a_0,a_1,\dots,a_n)=1$. We also denote $\deg(\gamma)=n$ and the height of $\gamma$ as $H(\gamma)=\max_{0\leq i\leq n}|a_i|$.
We order algebraic numbers by degree and height. Denote
$$\A_n(H)=\{\gamma\in\overline{\Q}\mid \deg(\gamma)=n,\;H(\gamma)\leq H\}.$$
Note that for any $n$, $H$ the set $\A_{n}(H)$ is finite. Let $\A_n=\bigcup_{H=1}^{\infty}\A_n(H)$. We define a probability distribution $\PPP_n$ on $\A_n$, so that given property $P$
$$\PPP_n[P]=\lim_{H\to\infty}\frac{\#\{\gamma\in\A_n(H)\mid \gamma\text{ satisfies }P\}}{\#A_n(H)}.$$
The following result was proved by van der Waerden \cite{Van Der Waerden}.
\begin{theorem}
[\cite{Van Der Waerden}]\label{Galois group Sn}
For any $n\geq 1$ we have
$$\PPP_n\big[\text{Galois closure of }\Q(\gamma)\text{ has Galois group }S_n\big]=1.$$
\end{theorem}

We show that given a number field $K$, $\Q(\gamma)$ will almost always intersect trivially with $K$.
\begin{theorem}\label{K cap Qgamma is Q}
Fix a number field field $K$ and positive integer $n$, then we have
$$\PPP_n[K\cap \Q(\gamma)=\Q]=1.$$
\end{theorem}

Drungilas, Dubickas and Jankauskas show in \cite{Z[gamma]cap Q}, that the ring $\Z[\gamma]\cap \Q$ is determined by the quantity $e(\gamma)=\gcd(a_1,a_2,\dots,a_n)$.
\begin{theorem}\cite[Theorem 5]{Z[gamma]cap Q}\label{Thm Z[gamma] cap Q}
Given an algebraic number $\gamma\in\overline{\Q}$, we have
$$\Z[\gamma]\cap\Q=\left\{\alpha\in\Q\mid \text{for all primes p, if }v_p(\alpha)<0 ,\text{ then } p\mid e(\gamma)\right\}.$$
\end{theorem}
\begin{corollary}
Given $\gamma\in\overline{\Q}$, we have $\Z[\gamma]\cap \Q=\Z$ if and only if $e(\gamma)=1$.
\end{corollary}

It is clear that all algebraic integers satisfy $e(\gamma)=1$. However, it is possible to have $e(\gamma)=1$ even if $\gamma$ is not an algebraic integer (for example $\gamma=\frac{1}{2+i}$).
We study the distribution of $e(\gamma)$, as $\gamma$ varies among algebraic numbers of degree $n$.
\begin{proposition}\label{prob e=k}
Given $n\geq 2$ and $k\geq 1$, we have
$$\PPP_n[e(\gamma)=k]=\frac{\zeta(n+1)}{\zeta(n)}\frac{\phi(k)}{k^{n+1}}.$$
\end{proposition}
Note that for $n\geq 2$, we have
$\frac{\zeta(n+1)}{\zeta(n)}\sum_{k=1}^{\infty}\frac{\phi(k)}{k^{n+1}}=1$.
For a number field $K$, we denote the ring of algebraic integers in $K$ as $\OO_K$ and the set of non-zero prime ideals of $\OO_K$ as $\Spec(\OO_K)$. Denote the minimal polynomial of $\gamma$ over $K$ as $f_{K,\gamma}(x)=b_nx^n+\dots+b_0\in K[x]$. So $f_{K,\gamma}(x)$ is irreducible in $K[x]$, $f_{K,\gamma}(\gamma)=0$ and $b_n=1$.
Singhal and Lin in \cite{Paper 1} study rings of the form $\OO_K[\gamma]\cap K$. For this they define certain finite sets of prime ideals $X(K,\gamma)\subseteq \Spec(\OO_K)$.

\begin{theorem}\cite[Theorem 1.3]{Paper 1}\label{X cond equiv}
Given a number ring $\OO_K$, an algebraic number $\gamma$ and a prime $\p\in\Spec(\OO_K)$, the following are equivalent:
\begin{itemize}
    \item For every $i\in [1,n]$: $v_{\p}(b_i)>v_{\p}(b_0)$.
    \item $\exists \alpha\in \p:\; \frac{1}{\alpha}\in\OO_K[\gamma]$.
    \item For every $\q\in\Spec(\OO_{K[\gamma]})$: if $\q\cap \OO_K=\p$ then $v_{\q}(\gamma)<0$.
\end{itemize}
\end{theorem}
Singhal and Lin \cite{Paper 1} define $X(K,\gamma)$ to be the set of prime ideals of $\Spec(\OO_K)$ that satisfy any of the equivalent conditions of Theorem \ref{X cond equiv}. They also note that
$$X(\Q,\gamma)=\{p\in\Spec(\Z)\mid p\mid e(\gamma)\}.$$
They characterise the ring $\OO_K[\gamma]\cap K$ in terms of $X(K,\gamma)$ as follows, thereby generalizing Theorem \ref{Thm Z[gamma] cap Q}.


\begin{theorem}\cite[Proposition 1.6]{Paper 1}\label{Cond for Ok[gamma]}
For $\alpha\in K$, we have $\alpha\in \OO_K[\gamma]$ if and only if
$$\{\p\in\Spec(\OO_K)\mid v_{\p}(\alpha)<0\}\subseteq X(K,\gamma).$$    
\end{theorem}

We call a subset $X\subseteq\Spec(\OO_K)$ stable, if given any $\p,\p'\in \OO_K$ with $\p\cap\Z=\p'\cap \Z$, either both $\p,\p'\in X$ or neither of them is in $X$. For example if $K=\Q[i]$ and $\gamma=\frac{1}{2+i}$ then $X(K,\gamma)=\{(2+i)\}$, which is not stable.

\begin{proposition}\label{Stable prob 1}
Given a number field $K$ and integer $n\geq 2$, we have
$$\PPP_n[X(K,\gamma)\text{ is stable }]=1.$$
\end{proposition}

We define constants $\alpha_{p,n}=\frac{1-\frac{1}{p^n}}{1-\frac{1}{p^{n+1}}}$ and $\beta_{p,n}=\frac{1}{p^n}\frac{1-\frac{1}{p}}{1-\frac{1}{p^{n}}}$. It should be noted that $0<\alpha_{p,n},\beta_{p,n}<1$, $\alpha_{p,n}(1+\beta_{p,n})=1$ and $\frac{\zeta(n+1)}{\zeta(n)}=\prod_{p}\alpha_{p,n}$.
For a given number field $K$, and finite stable subset $X\subseteq \Spec(\OO_K)$ we compute the probability that $X(K,\gamma)=X$.

\begin{proposition}\label{Prob of X(K,gamma)}
Given $n\geq 2$, a number field $K$ and a finite, stable subset $X\subseteq \Spec(\OO_K)$. Let $X_1=\{\p\cap\Z\mid \p\in X\}$, then we have
$$\PPP_n\big[X(K,\gamma)=X\big]
    =\frac{\zeta(n+1)}{\zeta(n)}\prod_{p\in X_1}\beta_{p,n}.$$
\end{proposition}
We also study the distribution of the rings $\OO_K[\gamma]\cap K$.
Note that if $k_1$ and $k_2$ are positive integers with the same prime factors then $\OO_K[\frac{1}{k_1}]=\OO_K[\frac{1}{k_2}]$.

\begin{theorem}\label{Ok[gamma] probability}
Given a number field $K$ and positive integers $n\geq 2$, $k\geq 1$, we have
$$\PPP_n\Big[\OO_K[\gamma]\cap K=\OO_K\Big[\frac{1}{k}\Big]\Big]
=\frac{\zeta(n+1)}{\zeta(n)}\prod_{p|k}\beta_{p,n}.$$
\end{theorem}
It should be noted that the probability does not depend on the number field $K$.

\begin{corollary}
Given a number field $K$ and $n\geq 2$, we have
$$\PPP_n\big[ \OO_K[\gamma]\cap K=\OO_K\big]=\frac{\zeta(n+1)}{\zeta(n)}.$$
\end{corollary}

Next, we show that given a rational prime $p$, the probability that all primes of $\OO_K$ above $p$ are in $X(K,\gamma)$ is $\alpha_{p.n}\beta_{p,n}$. Moreover, for distinct rational primes, these events are independent.
\begin{proposition}\label{Prop prob X subset X(K gamma)}
Given a number field $K$, a finite subset $X\subseteq \Spec(\OO_K)$ and $n\geq 2$.
Let $Y=\{\p\cap\Z\mid \p\in X\}$. Then
$$\PPP_n[X\subseteq X(K,\gamma)]=\prod_{p\in Y}\alpha_{p,n}\beta_{p,n}.$$
\end{proposition}

For a number field $K$ and rational prime $p$, we denote the number of primes of $\OO_K$ above $p$ as $r_K(p)$.
We study the statistics of the size of the set $X(K,\gamma)$.

\begin{proposition}\label{Prop exp and var of size of X(K,gamma)}
Given a number field $K$ and $n\geq 2$, we have
$$\E_n[|X(K,\gamma)|]=\sum_{p}\alpha_{p,n}\beta_{p,n}r_K(p),\;\;\text{Var}_n\;[|X(K,\gamma)|]=\sum_{p}\alpha_{p,n}\beta_{p,n}(1-\alpha_{p,n}\beta_{p,n}) r_K(p)^2.$$
\end{proposition}

\begin{theorem}\label{Thm bounds on prob size X(K, gamma) is t}
Given a number field $K$ and $n\geq 2$, let $\mu=\lfloor\frac{d}{n}\rfloor$. Given $\epsilon>0$, for sufficiently large $t$, we have
$$t^{-(1+\epsilon)nt} < \PPP_n[|X(K,\gamma)|=t]< t^{-\frac{1}{\mu+1+\epsilon}t}.$$
\end{theorem}

Singhal and Lin \cite{Paper 1} show that the class group of $\OO_K[\gamma]\cap K$ is determined by the subgroup of $\Cl(\OO_K)$ generated by $[\p]$ for $\p\in X(K,\gamma)$ as follows.

\begin{proposition}\cite[Proposition 1.7]{Paper 1}\label{Class group OK[gamma]cap K}
If $K$ is a number field and $\gamma\in\overline{\Q}$, then $\OO_K[\gamma]\cap K$ is a Dedekind domain and its class group is
$$\Cl(\OO_{K}[\gamma]\cap K)\cong \Cl(\OO_K)/\langle[\p]\mid \p\in X(K,\gamma)\rangle.$$
\end{proposition}

We study the subgroup $\big\langle [\p],\p\in X(K,\gamma)\big\rangle$ of $\Cl(\OO_K)$ for quadratic imaginary fields $K$.

\begin{theorem}\label{Thm: <X K gamma>=G}
Suppose $G$ is a finite abelian group which is not 2-torsion (in particular, $G$ is not trivial). Fix $n\geq 2$. For each quadratic imaginary field $K$, denote its discriminant as $d_K$. Then we have
$$\lim_{d_K\to-\infty}\PPP_n\big[ \big\langle [\p],\p\in X(K,\gamma)\big\rangle \cong G\big]=0.$$
\end{theorem}

Denote $\A_m'(H)=\{\alpha\in\Zbar\mid\deg(\alpha)=m,H(\alpha)\leq H\}$.
For each $\alpha\in \A'_m(H)$, $p$ splits into $r_{\Q(\alpha)}(p)$ primes in $\OO_{\Q(\alpha)}$, with $1\leq r_{\Q(\alpha)}\leq m$. Given $i\in [1,m]$, we consider the question of how often is $r_{\Q(\alpha)}=i$, as we sample over $\alpha\in \A'_m(H)$. Denote
$$g(m,i,p)=\lim_{H\to\infty}\frac{\#\{\alpha\in\A'_m(H)\mid r_{\Q(\alpha)}(p)=i\}}{\#\A'_m(H)}.$$
Further, denote $g(m,i)=\lim_{p\to\infty}g(m,i,p)$.
\begin{theorem}\label{Thm: prob rkp=i}
Given $m$ and $i$ such that $1\leq i\leq m$, $g(m,i)$ is the coefficient of $y^i$ in $\frac{\prod_{j=0}^{m-1}(y+j)}{m!}$.
\end{theorem}

The paper is organized as follows. In Section \ref{Sec: Base ring Z}, we prove Proposition \ref{prob e=k} and special cases of Proposition~\ref{Prob of X(K,gamma)}, Theorem~\ref{Ok[gamma] probability} and Proposition~\ref{Prop prob X subset X(K gamma)} when $K=\Q$. In Section \ref{Sec: Prob galois theory}, we prove Theorem \ref{K cap Qgamma is Q}. In Section \ref{Sec: Base ring OK}, we use the results from the previous two sections to prove Proposition~\ref{Stable prob 1} and complete the proofs of Proposition~\ref{Prob of X(K,gamma)}, Theorem~\ref{Ok[gamma] probability} and Proposition~\ref{Prop prob X subset X(K gamma)}. In Section~\ref{Sec: size of XKgamma}, we study the statistics of the size of $X(K,\gamma)$ and prove Proposition \ref{Prop exp and var of size of X(K,gamma)} and Theorem \ref{Thm bounds on prob size X(K, gamma) is t}. In Section~\ref{Sec: group gen by XKgamma}, we study the subgroup of the class group generated by $X(K,\gamma)$ for quadratic imaginary fields $K$ and prove Theorem~\ref{Thm: <X K gamma>=G}.
Finally in Section \ref{Sec: rkp}, we prove Theorem \ref{Thm: prob rkp=i}.

\section{Base ring $\Z$.}\label{Sec: Base ring Z}

Denote
$$\LL_{n}(H)=\left\{A(x)\in\Z[x]\mid \deg(A(x))=n, H(A(x))\leq H, A(x)\text{ is irreducible in }\Z[x]\right\}.$$
Polya and Szego show in \cite{irreducible count} that
\begin{equation*}
    \# \LL_n(H)=\frac{(2H)^{n+1}}{\zeta(n+1)}+O_n(H^n\log^2(H)).
\end{equation*}
The notation $O_n$ means that the implied constant only depends on $n$.
It is clear that $\#\A_n(H)=\frac{n}{2}\LL_n(H)$. Therefore
$$\# \A_n(H)=\frac{n}{2}\frac{(2H)^{n+1}}{\zeta(n+1)}+O_n(H^n\log^2(H)).$$


\begin{lemma}\label{lem count e in Y}
Given $n\geq 2$, and a subset $Y\subseteq \N_{>0}$, we have
$$\#\{\gamma\in\A_n(H)\mid e(\gamma)\in Y\}=\frac{n}{2}\frac{(2H)^{n+1}}{\zeta(n)}\sum_{k\in Y\;k\leq H}\frac{\phi(k)}{k^{n+1}} +O_n\left(H^n \log^2(H) \right).$$
\end{lemma}
\begin{proof}
Denote
$$\CC_n(H)=\{(c_1,\dots,c_n)\in\Z^n\mid \gcd(c_1,\dots,c_n)=1, |c_i|\leq H\}.$$
Nymann in \cite{k int co prime} proves that for $n\geq 2$,
\begin{equation*}
    \# \CC_n(H)=\frac{(2H)^{n}}{\zeta(n)}+O_n(H^{n-1}\log(H)).
\end{equation*}
Denote
$$\CC_n(k,H)=\{c_0+\dots+c_nx^n\in\Z[x]\mid \gcd(c_0,c_1,\dots,c_n)=1,\gcd(c_1,\dots,c_n)=k, |c_i|\leq H\}.$$
We see that
\begin{equation*}
    \begin{split}
    \#\CC_n(k,H)&=\#\{c_0\in[-H,H]\mid \gcd(c_0,k)=1\}\#\CC_n\left(\frac{H}{k}\right)\\
    &= \left(\frac{\phi(k)}{k}2H+O(k)\right)\left(\frac{(2H)^{n}}{k^n\zeta(n)}+O_n\left(\frac{H^{n-1}}{k^{n-1}} \log(H)\right)\right)\\
    &=\frac{\phi(k)}{k^{n+1}}\frac{(2H)^{n+1}}{\zeta(n)}+O_n\left( \frac{H^{n}}{k^{n-1}}\log(H)+\frac{H^{n-1}}{k^{n-2}}\log(H)\right).
    \end{split}
\end{equation*}
Next, we have
\begin{equation*}
    \begin{split}
   \#\{\gamma\in\A_n(H)\mid e(\gamma)\in Y\}
    &=\sum_{k\in Y}\frac{n}{2}\#\left(\CC_n(k,H)\cap\LL_n(H)\right)\\
    &=\frac{n}{2}\sum_{k\in Y}\#\CC_n(k,H) -\frac{n}{2}\#\bigg(\Big(\bigcup_{k\in Y}\CC_n(k,H)\Big)\setminus \LL_n(H)\bigg).
    \end{split}
\end{equation*}
Next,
\begin{equation*}
    \begin{split}
    \#\bigg(\Big(\bigcup_{k\in Y}\CC_n(k,H)\Big)\setminus \LL_n(H)\bigg)
    &\leq \#\CC_{n+1}(H)- \#\LL_n(H)\\
    &=\frac{(2H)^{n+1}}{\zeta(n+1)}+O_n(H^{n})-\frac{(2H)^{n+1}}{\zeta(n+1)}-O_n(H^n\log^2(H))\\
    &=O_n\left(H^n \log^2(H)\right).
    \end{split}
\end{equation*}
Finally, notice that if $k>H$, then $C_n(k,H)=\emptyset$. Therefore, we have
\begin{equation*}
    \begin{split}
    &\#\{\gamma\in\A_n(H)\mid e(\gamma)\in Y\}=\frac{n}{2}\frac{(2H)^{n+1}}{\zeta(n)}\sum_{k\in Y\;k\leq H}\frac{\phi(k)}{k^{n+1}}\\
    &+O_n\Big(H^n \log^2(H) +H^n\log(H)\sum_{k\in Y\;k\leq H}\frac{1}{k^{n-1}} +H^{n-1}\log(H)\sum_{k\in Y\;k\leq H}\frac{1}{k^{n-2}} \Big)\\
    &=\frac{n}{2}\frac{(2H)^{n+1}}{\zeta(n)}\sum_{k\in Y\;k\leq H}\frac{\phi(k)}{k^{n+1}} +O_n\left(H^n \log^2(H)\right). \qedhere
    \end{split}
\end{equation*}
\end{proof}

\begin{corollary}\label{Cor: Prob e in Y}
Given $n\geq 2$, and a subset $Y\subseteq \N_{>0}$, we have
$$\PPP_n[e(\gamma)\in Y]=\frac{\zeta(n+1)}{\zeta(n)}\sum_{k\in Y}\frac{\phi(k)}{k^{n+1}}.$$
\end{corollary}
\begin{proof}
\begin{equation*}
    \begin{split}
    \PPP_n\left[e(\gamma)\in Y\right]
    &=\lim_{H\to\infty} \frac{\#\left\{\gamma\in \A_n(H)\mid e(\gamma)\in Y\right\}}{\#\A_n(H)}  \\
    &= \lim_{H\to\infty}\frac{\frac{n}{2}\frac{(2H)^{n+1}}{\zeta(n)}\sum_{k\in Y\;k\leq H}\frac{\phi(k)}{k^{n+1}} +O_n\left(H^{n}\log^2(H) \right)}{\frac{n}{2}\frac{(2H)^{n+1}}{\zeta(n+1)}+O_n(H^n\log^2(H))}  \\
    &=\frac{\zeta(n+1)}{\zeta(n)}\lim_{H\to\infty}\sum_{k\in Y\;k\leq H} \frac{\phi(k)}{k^{n+1}} =\frac{\zeta(n+1)}{\zeta(n)}\sum_{k\in Y} \frac{\phi(k)}{k^{n+1}}.\qedhere
    \end{split}
\end{equation*}
\end{proof}

\begin{customprop}{\ref{prob e=k}}
Given $n\geq 2$ and $k\geq 1$,
$$\PPP_n[e(\gamma)=k]=\frac{\zeta(n+1)}{\zeta(n)}\frac{\phi(k)}{k^{n+1}}.$$
\end{customprop}
\begin{proof}
It follows from Corollary \ref{Cor: Prob e in Y}.
\end{proof}


Let $\Pf(\Spec(\Z))$ be the set of all finite subsets of $\Spec(\Z)$.

\begin{lemma}\label{Lem X(Q gamma) in fancy X}
Given $n\geq 2$ and $\X\subseteq \Pf(\Spec(\Z))$, we have
$$\PPP_n[X(\Q,\gamma)\in \X]=\frac{\zeta(n+1)}{\zeta(n)}\sum_{X\in \X}\prod_{p\in X}\beta_{p,n}.$$
\end{lemma}
\begin{proof}
Let $Y=\{k\in \Z_{>0}\mid \{p\mid p\mid k\}\in \X\}$. Since $X(\Q,\gamma)=\{p\mid p\mid e(\gamma)\}$, we know that $X(\Q,\gamma)\in\X$ if and only if $e(\gamma)\in Y$. Therefore,
\begin{equation*}
    \begin{split}
    \PPP_n[X(\Q,\gamma)\in \X]
    &= \PPP_n[e(\gamma)\in Y]
    =\frac{\zeta(n+1)}{\zeta(n)}\sum_{k\in Y}\frac{\phi(k)}{k^{n+1}}\\
    &=\frac{\zeta(n+1)}{\zeta(n)} \sum_{X\in\X}\sum_{\substack{k\geq 1\\ \{p\mid\; p|k\}=X}}\frac{\phi(k)}{k^{n+1}}
    =\frac{\zeta(n+1)}{\zeta(n)} \sum_{X\in\X}\prod_{p\in X}\left(\sum_{t=1}^{\infty}\frac{\phi(p^t)}{p^{t(n+1)}}\right)\\
    &=\frac{\zeta(n+1)}{\zeta(n)} \sum_{X\in\X}\prod_{p\in X}\beta_{p,n}.\qedhere
    \end{split}
\end{equation*}
\end{proof}

\begin{corollary}\label{prob of X(Q,gamma)}
Given a finite subset $X\subseteq \Spec(\Z)$ and $n\geq 2$, we have
$$\PPP_n[X(\Q,\gamma)=X]=\frac{\zeta(n+1)}{\zeta(n)}\prod_{p\in X}\beta_{p,n}.$$
\end{corollary}

\begin{proposition}\label{Prob Z[gamma]cap Q}
Given $n\geq 2$ and $k\geq 1$, we have
$$\PPP_n\Big[\Z[\gamma]\cap\Q=\Z\Big[\frac{1}{k}\Big]\Big]
=\frac{\zeta(n+1)}{\zeta(n)}\prod_{p|k}\beta_{p,n}.$$    
\end{proposition}
\begin{proof}
Let $X=\{p\mid p\mid k\}$. Now, by Theorem \ref{Thm Z[gamma] cap Q}, we know that $\Z[\gamma]\cap \Q=\Z[\frac{1}{k}]$ if and only if $X(\Q,\gamma)=X$. Therefore, by Corollary \ref{prob of X(Q,gamma)}, we have
\[\PPP_n\Big[\Z[\gamma]\cap \Q=\Z\Big[\frac{1}{k}\Big]\Big]
    =\PPP_n\left[X(\Q,\gamma)=X\right]
    =\frac{\zeta(n+1)}{\zeta(n)}\prod_{p|k}\beta_{p,n}.\qedhere\]
\end{proof}

\begin{proposition}\label{prop p in X(Q gamma) independent}
Given $n\geq 2$ and a finite set of primes $X\subseteq \Spec(\Z)$, we have
$$\PPP_n[X\subseteq X(\Q,\gamma)]=\prod_{p\in X}\alpha_{p,n}\beta_{p,n}.$$
\end{proposition}
\begin{proof}
Let $\X=\{Y\in\Pf(\Spec (\Z))\mid X\subseteq Y\}$. Therefore, by Lemma \ref{Lem X(Q gamma) in fancy X}, we have
\begin{equation*}
    \begin{split}
    \PPP_n[X\subseteq X(\Q,\gamma)]
    &=\PPP_n[X(\Q,\gamma)\in \X]
    =\frac{\zeta(n+1)}{\zeta(n)}\sum_{Y\in \X}\prod_{p\in Y}\beta_{p,n}\\
    &= \frac{\zeta(n+1)}{\zeta(n)} \prod_{p\in X}\beta_{p,n} \prod_{p\notin X}(1+\beta_{p,n})\\
    &=\prod_{p}\alpha_{p,n} \prod_{p\in X}\beta_{p,n} \prod_{p\notin X}\alpha_{p,n}^{-1}
    =\prod_{p\in X}\alpha_{p,n}\beta_{p,n}. \qedhere
    \end{split}
\end{equation*}
\end{proof}

\begin{proposition}\label{X(Q,gamma) subset X}
Given $n\geq 2$ and a set of primes $X\subseteq \Spec(\Z)$, we have
$$\PPP_n[X(\Q,\gamma)\subseteq X]=\prod_{p\notin X}\alpha_{p,n}.$$
\end{proposition}
\begin{proof}
Let $\X=\{Y\in\Pf(\Spec (\Z))\mid Y\subseteq X\}$. Therefore, by Lemma \ref{Lem X(Q gamma) in fancy X}, we have
\begin{equation*}
    \begin{split}
    \PPP_n[X(\Q,\gamma)\subseteq X]
    &=\PPP_n[X(\Q,\gamma)\in \X]
    =\frac{\zeta(n+1)}{\zeta(n)}\sum_{Y\in \X}\prod_{p\in Y}\beta_{p,n}\\
    &= \frac{\zeta(n+1)}{\zeta(n)} \prod_{p\in X}(1+\beta_{p,n})
    =\prod_{p}\alpha_{p,n} \prod_{p\in X}\alpha_{p,n}^{-1}
    =\prod_{p\notin X}\alpha_{p,n}^{-1}. \qedhere
    \end{split}
\end{equation*}
\end{proof}

\section{Probabilistic Galois Theory}\label{Sec: Prob galois theory}

In this section we will be proving Theorem \ref{K cap Qgamma is Q}.
It is a well known fact that for $n\neq 4$, the only normal subgroups of $S_n$ are $\{id\}$, $A_n$ and $S_n$. From this it is easy to see that for any $n\geq 1$ and $2<k<n$, $S_n$ has no subgroup of index $k$. Moreover $A_n$ is the only subgroup of $S_n$ of index $2$.

\begin{lemma}\label{n> deg+1}
Given a number field $K$ and integer $n\geq [K:\Q]+1$, we have
$$\PPP_n[K\cap \Q(\gamma)=\Q]=1.$$
\end{lemma}
\begin{proof}
Consider $\gamma\in \A_n$, such that Galois closure of $\Q(\gamma)$ has Galois group $S_n$. Let $\gamma_2,\dots,\gamma_n$ be the conjugates of $\gamma$ over $\Q$, so $L=\Q(\gamma,\gamma_2,\dots,\gamma_n)$ is the Galois closure of $\Q(\gamma)$. Let $H$ be the Galois group of $L$ over $K\cap \Q(\gamma)$, so $H\subseteq S_n$. Now, $K\cap \Q(\gamma)\subseteq \Q(\gamma)$, so $S_{n-1}\subseteq H$. This means that the index of $H$ is at most $n$. Therefore, index of $H$ is in $\{1,2,n\}$.
\begin{itemize}
    \item If the index of $H$ is $n$, then $H=S_{n-1}$, so $K\cap \Q(\gamma)=\Q(\gamma)$. This means that $Q(\gamma)\subseteq K$. But this is impossible since $[\Q(\gamma):\Q]=n>[K:\Q]$.
    \item If the index of $H$ is $2$, then $H=A_n$, but this is impossible since $S_{n-1}\subseteq H$ and $S_{n-1}\not\subseteq A_{n}$.
\end{itemize}
Therefore, $H$ has index $1$, which means $H=S_n$ and $K\cap \Q[\gamma]=\Q$.
Now, we have
\[1\geq \PPP_n[K\cap\Q(\gamma)=\Q]\geq \PPP_n\big[\text{Galois closure of }\Q(\gamma)\text{ has Galois group }S_n\big]=1.\qedhere\]
\end{proof}

We have proven the special case of Theorem \ref{K cap Qgamma is Q} when $[K:\Q]<n$. We will need the following lemma to complete the proof of Theorem \ref{K cap Qgamma is Q}.
\begin{lemma}\label{Lemma: Disc=my^2}
Given $n\geq 2$ and a square free integer $m$ ($m$ can be negative), we have
$$\PPP_n[\exists y\in\Z: \disc(F_{\gamma})=my^2]=0.$$
\end{lemma}

We will be using the following results from \cite{Pila} and \cite{Hilbert Irr} to prove Lemma \ref{Lemma: Disc=my^2}.

\begin{lemma}[\cite{Pila}]\label{Pila}
Given $\epsilon>0$ and $d\in\N_{\geq 1}$, there is a constant $c(d,\epsilon)$ such that
if $F$ is a subset of an irreducible, plane algebraic curve of degree $d$ inside a square of side
$H$, then the number of lattice points on $F$ is bounded by $c(d,\epsilon)H^{\frac{1}{d}+\epsilon}$.
\end{lemma}

\begin{lemma}[\cite{Hilbert Irr}]\label{Hilbert Irr bound}
Let $f(X,t_1,\dots,t_s)\in \Z[X,t_1,\dots,t_s]$ be an irreducible polynomial. Then
\begin{equation*}
    \begin{split}
    &\#\big\{(\alpha_1,\dots,\alpha_s)\in \Z^s\mid |\alpha_i|\leq H, f(X,\alpha_1,\dots,\alpha_s)\text{ is not irreducible in } \Z[X]\big\}\\
    &=O(H^{s-\frac{1}{2}}\log(H)).    
    \end{split}
\end{equation*}
\end{lemma}

Let $A_0,A_1,\dots,A_n$ be indeterminates and $f(x)=A_0x^n+A_{1}x^{n-1}+\dots+A_n$.
It is a well known fact that $\disc(f)(A_0,A_1,\dots,A_n)$ is an irreducible, homogeneous polynomial of degree $2n-2$ in $\Z[A_0,\dots,A_n]$.
\begin{lemma}\label{An-1^n}
Express $\disc(f)(A_0,\dots, A_n)$ as a linear combination of monomials in $A_0,\ldots, A_n$. Consider $t_0,\dots,t_{n-2},t_{n}\in \Z_{\geq 0}$ with $t_0+\dots+t_{n-2}+t_n=n-2$. If $t_0\neq n-2$, then the coefficient of $A_{0}^{t_0}\dots A_{n-2}^{t_{n-2}}A_n^{t_n}A_{n-1}^{n}$ is zero. Moreover, the coefficient of $A_{n-1}^n A_0^{n-2}$ is $\pm (n-1)^{n-1}$.
\end{lemma}
\begin{proof}
It is known that $\disc(f)(A_0,\dots,A_n)$ can be expressed in terms of the Sylvester matrix as follows:
$$\disc(f)(A_0,\dots,A_n)=\frac{(-1)^{\frac{n(n-1)}{2}}}{A_n}\det(S(f(x),f'(x))).$$
The Sylvester matrix is a $2n-1\times 2n-1$ matrix. Its non-zero entries are as follows:
\begin{equation*}
    a_{ij}=\begin{cases}
        A_{n+i-j} &\text{if } 1\leq i\leq n-1, i\leq j\leq i+n,\\
        (i+1-j)A_{i+1-j} &\text{if } n\leq i\leq 2n-1, i-(n-1)\leq j\leq i.
    \end{cases}
\end{equation*}
Below we write the Sylvester matrix for $n=5$.
\[S(f(x),f'(x))=
\begin{bmatrix}
A_5 &A_{4}& A_{3} &A_{2} &A_{1} &A_{0} & 0&0 &0\\
0 &A_5 &A_{4}&A_{3} &A_{2} & A_{1}& A_{0}& 0  &0\\
0&0 &A_5 &A_{4}&A_{3} &A_{2} & A_{1}& A_{0}& 0\\
0&0&0 &A_5 &A_{4}&A_{3} &A_{2} & A_{1}& A_{0}\\
5A_5 &4A_{4}& 3A_{3} &2A_{2} &1A_{1} &0 & 0&0 &0\\
0 &5A_5 &4A_{4}&3A_{3} &2A_{2} & 1A_{1}& 0& 0  &0\\
0&0 &5A_5 &4A_{4}&3A_{3} &2A_{2} & 1A_{1}& 0& 0\\
0&0&0 &5A_5 &4A_{4}&3A_{3} &2A_{2} & 1A_{1}& 0\\
0&0&0&0 &5A_5 &4A_{4}&3A_{3} &2A_{2} & 1A_{1}\end{bmatrix}.\]
Note that there are $n-1$ rows in the top part of the matrix and $n$ rows in the bottom part.
We want to compute the coefficient of $A_0^{t_0}\dots A_{n-2}^{t_{n-2}}A_n^{1+t_{n}}A_{n-1}^{n}$ in this determinant. The determinant is expanded by picking $2n-1$ entries at a time such that we get an entry from each row and each column. We have $A_{n-1}$ in columns $2$ to $n+1$, and in no other columns. Therefore, in order to obtain $A_0^{t_0}\dots A_{n-2}^{t_{n-2}}A_n^{1+t_{n}}A_{n-1}^{n}$, we must pick a $A_{n-1}$ in the $j^{th}$ column for each $2\leq j\leq n+1$.

We claim that in order to obtain $A_0^{t_0}\dots A_{n-2}^{t_{n-2}}A_n^{1+t_{n}}A_{n-1}^{n}$, for each $j$ in $[3,n+1]$, in the $j^{th}$ column we must pick the $(n-1)A_{n-1}$ in the bottom part of the matrix and for each $j$ in $[3,n]$, in the $(j+n-1)^{th}$ column we must pick the $A_0$ in the top part of the matrix. We prove this by backward induction on $j$. For $j=n+1$, note that the $(n+1)^{th}$ column only has one $A_{n-1}$ in it, $a_{2n-1,n+1}=(n-1)A_{n-1}$. So in the $(n+1)^{th}$ column, we must pick the $A_{n-1}$ in the bottom part.

Next, assume the induction hypothesis that for some $j\in [3,n]$, we know that for each $k\in [j+1,n+1]$, in the $k^{th}$ column we must pick the $A_{n-1}$ in the bottom part of the matrix and for each $k\in [j+1,n]$, in the $(k+n-1)^{th}$ column we must pick the $A_0$ in the top part of the matrix. Now, the non-zero entries in the $(j+n-1)^{th}$ column are $a_{t,j+n-1}$ for $j-1\leq t\leq n-1$ and $j+n-1\leq t\leq 2n-1$. However, for $j\leq t\leq n-1$ we have already chosen $a_{t,t+1+(n-1)}=A_0$. For $j+n-1\leq t\leq 2n-1$, we have already chosen $a_{t,t+2-n}=(n-1)A_{n-1}$. Therefore, in the $(j+n-1)^{th}$ column we must choose $a_{j-1,j+n-1}=A_0$ (it is in the top part).
Next, in the $j^{th}$ column, there are two $A_{n-1}$: $a_{j-1,j},a_{n+j-2,j}$. However, we have already chosen $A_0=a_{j-1,j+n-1}$. Therefore, we must choose $a_{n+j-2,j}=(n-1)A_{n-1}$ (it is in the bottom part). This completes the induction step.

We have shown that for $3\leq j\leq n+1$, we must pick $a_{n+2-j,j}=(n-1)A_{n-1}$ and for $3\leq j\leq n$, we must choose $a_{j-1,j-1+n}=A_0$. This means that columns $3$ to $2n-1$ contribute $(n-1)^{n-1}A_0^{n-2}A_{n-1}^{n-1}$. Now, columns $1,2$ and rows $1,n$ are left. The entries are $a_{1,1}=A_n$, $a_{1,2}=A_{n-1}$, $a_{n,1}=nA_n$ and $a_{n,2}=(n-1)A_{n-1}$. The contribution from here is $-A_{n}A_{n-1}$.

Therefore, if $t_0\neq n-2$, then the coefficient of $A_0^{t_0}\dots A_{n-2}^{t_{n-2}}A_n^{1+t_{n}}A_{n-1}^{n}$ in this determinant is $0$. Whereas, the coefficient of $A_0^{n-2}A_n^{1}A_{n-1}^{n}$ is $-(n-1)^{n-1}$. The result follows.
\end{proof}

\begin{proof}[Proof of Lemma \ref{Lemma: Disc=my^2}]
We know that $\disc(f)(A_0,\dots,A_n)$ is a homogeneous polynomial of degree $2n-2$. Therefore there is a constant $C$ (which only depends on $n$) such that for all $\alpha_0,\dots,\alpha_n\in \Z$ with $|\alpha_i|\leq H$, we have $|\disc(f)(\alpha_0,\dots,\alpha_n)|\leq C H^{2n-2}$. Thus if $\gamma\in A_n(H)$ and $\disc(F_{\gamma})=my^2$ then $|y|\leq \sqrt{C} H^{n-1}$.

Let $d=\deg_{A_{n-1}}(\disc(f))$.
We know from Lemma \ref{An-1^n} that $d\geq n$. Let $I(\Z[A_{n-1}])$ be the set of irreducible polynomials in $\Z[A_{n-1}]$. By Lemma \ref{Hilbert Irr bound}, we know that
\begin{equation*}
    \begin{split}
    \#&\big\{(\alpha_1,\dots,\alpha_n)\in \Z^n\mid |\alpha_i|\leq H, \disc(f)(\alpha_1,\dots,\alpha_{n-1},A_{n-1},\alpha_{n})\notin I(\Z[A_{n-1}])\big\}\\
    &=O(H^{n-\frac{1}{2}}\log(H)).
    \end{split}
\end{equation*}
Given $(\alpha_1,\dots,\alpha_n)\in[-H,H]^n$ such that $\disc(f)(\alpha_1,\dots,\alpha_{n-1},A_{n-1},\alpha_{n})\in I(\Z[A_{n-1}])$, we see that $my^2-\disc(f)(\alpha_1,\dots,\alpha_{n-1},A_{n-1},\alpha_{n})$ is an irreducible curve in variables $y,A_{n-1}$. Pick $0<\epsilon<\frac{1}{2n(n-1)}$. The number of $\alpha\in \Z$ with $|\alpha|\leq H$, such that $\disc(f)(\alpha_1,\dots,\alpha_{n-1},\alpha,\alpha_n)=my^2$ has a solution (for $y$) is at most the number of solutions to $my^2-\disc(f)(\alpha_1,\dots,\alpha_{n-1},A_{n-1},\alpha_{n})$ with $|A_{n-1}|,|y|\leq \sqrt{C}H^{n-1}$. By Lemma \ref{Pila}, this is at most
$$c(d,\epsilon)\left(\sqrt{C}H^{n-1}\right)^{\frac{1}{d}+\epsilon} =O(H^{\frac{n-1}{d}+\epsilon(n-1)})=O(H^{1-\frac{1}{2n}}).$$
Now,
\begin{equation*}
    \begin{split}
    \#&\{\gamma\in A_{n}(H)\mid \exists y\in\Z:\disc(F_{\gamma}(x))=my^2\}\\
    &\leq n\#\{g(x)=a_nx^n+\dots+a_0\in\Z[x]\mid |a_i|\leq H,\exists y\in\Z:\disc(g(x))=my^2\}\\
    &\leq n\#\{(\alpha_1,\dots,\alpha_n)\in\Z^{n}\mid |\alpha_i|\leq H, \disc(f)(\alpha_1,\dots,\alpha_{n-1},A_{n-1},\alpha_{n})\notin I(\Z[A_{n-1}])\}(2H+1)\\
    &+n\sum_{\substack{(\alpha_1,\dots,\alpha_n)\in\Z^{n} \\|\alpha_i|\leq H \\\disc(f)(\alpha_1,\dots,A_{n-1},\alpha_{n})\in I(\Z[A_{n-1}])}}\#\{\alpha\in\Z\mid |\alpha|\leq H, \exists y\in \Z: \disc(f)(\alpha_1,\dots,\alpha_{n-1},\alpha,\alpha_n)=my^2\}\\
    &\leq n O\left(H^{n-\frac{1}{2}}\log(H)\right)(2H+1) +n (2H+1)^n O(H^{1-\frac{1}{2n}})\\
    &=O(H^{n+1-\frac{1}{2n}}).
    \end{split}
\end{equation*}
Therefore,
\begin{align*} 
    \PPP_n[\exists y\in\Z: \disc(F_{\gamma})=my^2]
    &=\lim_{H\to\infty} \frac{\#\{\gamma\in A_{n}(H)\mid \exists y\in\Z:\disc(F_{\gamma}(x))=my^2\}}{\#\A_n(H)}\\
    &\leq \lim_{H\to\infty} \frac{O(H^{n+1-\frac{1}{2n}})}{\frac{n}{2}\frac{(2H)^{n+1}}{\zeta(n+1)}+O_n(nH^n\log^2(H))}=0.  \qedhere
    \end{align*}
\end{proof}

\begin{customthm}{\ref{K cap Qgamma is Q}}
Given a fixed field $K$ and positive integer $n$, we have
$$\PPP_n[K\cap \Q(\gamma)=\Q]=1.$$
\end{customthm}
\begin{proof}
Fix $n$. We induct on $[K:\Q]$.
The base cases $[K:\Q]\leq n-1$ are provided by Lemma \ref{n> deg+1}.

Now consider some $K$ with $[K:\Q]\geq n$. From the induction hypothesis we know that for all proper subfields $K'\subsetneq K$, we have $\PPP_n[K'\cap \Q(\gamma)=\Q]=1$.
This implies that
$$\PPP_n[K\cap \Q(\gamma)=\Q]=1-\PPP_n[K\cap \Q(\gamma)=K].$$
Note that if $K\cap \Q(\gamma)=K$, then $K\subseteq \Q(\gamma)$. However $[\Q(\gamma):\Q]=n$ and $[K:\Q]\geq n$. This forces $K=\Q(\gamma)$.
This means that
$$\PPP_n[K\cap \Q(\gamma)=\Q]=1-\PPP_n[K=\Q(\gamma)].$$

Now if $[K:\Q]\geq n+1$, then $\PPP_n[K=\Q(\gamma)]=0$ and we are done. Therefore suppose $[K:\Q]=n$.

\begin{itemize}
    \item Case 1: The Galois closure of $K$ does not have Galois group $S_n$. We know by Proposition \ref{Galois group Sn}, that
    $$\PPP_n\big[\text{Galois closure of }\Q(\gamma)\text{ has Galois group }S_n\big]=1.$$
    Therefore $\PPP_n[K=\Q(\gamma)]=0$.
    \item Case 2: The Galois closure of $K$ does have Galois group $S_n$. Say Galois closure is $L$. The unique subgroup of index $2$ in $S_n$ is $A_n$ and hence $L$ has a unique quadratic subfield $K_1$. Say $K_1=\Q(\sqrt{m})$, for some square-free $m\in\Z$. Now if $K=\Q(\gamma)$ for some $\gamma\in \A_n$. Then $K_1=\Q(\sqrt{\disc(F_{\gamma})})$. Therefore
    $$\PPP_n[K=\Q(\gamma)]\leq \PPP_n[\exists y\in\Z: \disc(F_{\gamma})=my^2]=0,$$
    where we applied Lemma \ref{Lemma: Disc=my^2} in the last step.
\end{itemize}
In both cases we see that $\PPP_n[K=\Q(\gamma)]=0$ and hence $\PPP_n[K\cap \Q(\gamma)=\Q]=1$. This completes the induction step and hence the proof.
\end{proof}

\section{Base ring $\OO_K$}\label{Sec: Base ring OK}
In this section we will prove Proposition \ref{Stable prob 1}, Proposition \ref{Prob of X(K,gamma)}, Theorem \ref{Ok[gamma] probability} and Proposition \ref{Prop prob X subset X(K gamma)}. We will need the following Lemma from \cite{Paper 1}.

\begin{lemma}\cite[Lemma 3,1]{Paper 1}\label{going up going down}
If $K\subseteq L$ are number fields, then
$$X(K,\gamma)=\{\p\in\Spec(\OO_K)\mid \text{for every }\q\in\Spec(\OO_L), \text{ if } \q\cap\OO_K=\p \text{ then } \q\in X(L,\gamma)\},$$
$$X(L,\gamma)\supseteq\{\q\in\Spec(\OO_L)\mid \q\cap\OO_K\in X(K,\gamma)\}.$$
Moreover if $L\cap K(\gamma)=K$, then we have
$$X(L,\gamma)=\{\q\in\Spec(\OO_L)\mid \q\cap\OO_K\in X(K,\gamma)\}.$$
\end{lemma}

\begin{customprop}{\ref{Stable prob 1}}
Given a number field $K$ and integer $n\geq 2$, we have
$$\PPP_n[X(K,\gamma)\text{ is stable }]=1.$$
\end{customprop}
\begin{proof}
From Lemma \ref{going up going down}, it follows that if $\Q(\gamma)\cap K=\Q$, then $X(K,\gamma)$ is stable. Therefore, we are done by Theorem \ref{K cap Qgamma is Q}.
\end{proof}

\begin{lemma}\label{Lem X(K gamma) in fancy X}
Given $\X\subseteq \Pf(\Spec(\OO_K))$, let
$$\Y=\Big\{Y\in \Pf(\Spec(\Z))\mid \{\p\in \Spec(\OO_K)\mid \p\cap\Z\in Y\}\in \X\Big\}.$$
Then, we have
$$\PPP_n[X(K,\gamma)\in \X]=\PPP_n[X(\Q,\gamma)\in\Y]=\frac{\zeta(n+1)}{\zeta(n)}\sum_{Y\in \Y}\prod_{p\in Y}\beta_{p,n}.$$
\end{lemma}
\begin{proof}
By Theorem \ref{K cap Qgamma is Q}, Lemma \ref{Lem X(Q gamma) in fancy X} and Lemma \ref{going up going down}, we know that
\begin{equation*}
    \begin{split}
    \PPP_n[X(K,\gamma)\in\X]
    &= \PPP_n[X(K,\gamma)\in\X \text{ and } \Q(\gamma)\cap K=\Q]\\
    &= \PPP_n[X(\Q,\gamma)\in\Y \text{ and } \Q(\gamma)\cap K=\Q]\\
    &=\PPP_n[X(\Q,\gamma)\in\Y]
    =\frac{\zeta(n+1)}{\zeta(n)}\sum_{Y\in \Y}\prod_{p\in Y}\beta_{p,n}.\qedhere
    \end{split}
\end{equation*}
\end{proof}

\begin{customprop}{\ref{Prob of X(K,gamma)}}
Given $n\geq 2$, a number field $K$ and a finite, stable subset $X\subseteq \Spec(\OO_K)$. Let $X_1=\{\p\cap\Z\mid \p\in X\}$, then we have
$$\PPP_n\big[X(K,\gamma)=X\big]
    =\frac{\zeta(n+1)}{\zeta(n)}\prod_{p\in X_1}\beta_{p,n}.$$
\end{customprop}
\begin{proof}
This follows from Lemma \ref{Lem X(K gamma) in fancy X}.
\end{proof}

\begin{proposition}\label{Prop prob X(K gamma) subset X}
Given a number field $K$, a subset $X\subseteq \Spec(\OO_K)$ and $n\geq 2$.
Let $Y=\{p\in \Spec(\Z)\mid \text{all primes of }\OO_K \text{ above }p \text{ are in } X\}$. Then
$$\PPP_n[X(K,\gamma)\subseteq X]=\prod_{p\notin Y}\alpha_{p,n}.$$
\end{proposition}
\begin{proof}
This follows from Lemma \ref{Lem X(K gamma) in fancy X} and Proposition \ref{X(Q,gamma) subset X}.
\end{proof}

\begin{customthm}{\ref{Ok[gamma] probability}}
Given a number field $K$ and integers $n\geq 2$, $k\geq 1$, we have
$$\PPP_n\big[\OO_K[\gamma]\cap K=\OO_K[1/k]\big]
=\frac{\zeta(n+1)}{\zeta(n)}\prod_{p|k}\beta_{p,n}.$$
\end{customthm}
\begin{proof}
Let $X=\{\p\in\Spec(\OO_K)\mid k\in\p\}$. Notice that $X$ is stable and hence,
$$\PPP_n[X(K,\gamma)=X]=\frac{\zeta(n+1)}{\zeta(n)}\prod_{p|k}\beta_{p,n}.$$
We will show that $X(K,\gamma)=X$ if and only if $\OO_K[\gamma]\cap K=\OO_K[\frac{1}{k}]$. Let $\gamma_1=\frac{1}{k}$, so $\OO_K[\gamma_1]\cap K=\OO_K[\frac{1}{k}]$. Since $f_{K,\gamma_1}(x)=x-\frac{1}{k}$, we see by Theorem \ref{X cond equiv} that $X(K,\gamma_1)=X$. Theorem \ref{Cond for Ok[gamma]} tells us that $\OO_{K}[\gamma]\cap K=\OO_K[\gamma_1]\cap K$ if and only if $X(K,\gamma)=X(K,\gamma_1)$. This means that $\OO_K[\gamma]\cap K=\OO_K[\frac{1}{k}]$ if and only if $X(K,\gamma)=X$.
\end{proof}

\begin{customprop}{\ref{Prop prob X subset X(K gamma)}}
Given a number field $K$, a finite subset $X\subseteq \Spec(\OO_K)$ and $n\geq 2$.
Let $Y=\{\p\cap\Z\mid \p\in X\}$. Then
$$\PPP_n[X\subseteq X(K,\gamma)]=\prod_{p\in Y}\alpha_{p,n}\beta_{p,n}.$$
\end{customprop}
\begin{proof}
This follows from Lemma \ref{Lem X(K gamma) in fancy X} and Proposition \ref{prop p in X(Q gamma) independent}.
\end{proof}

\section{Statistics of the size of $X(K,\gamma)$}\label{Sec: size of XKgamma}

Denote $a_{K,n,t}=\PPP_n\left[ |X(K,\gamma)|=t\right]$.
Recall that $r_K(p)$ is the number of primes of $\OO_K$ above $p$.

\begin{proposition}\label{Prop prob size X(K,gamma) is t}
Given a number field $K$, integers $n\geq 2$ and $t\geq 0$, we have
$$\PPP_n[|X(K,\gamma)|=t]=\frac{\zeta(n+1)}{\zeta(n)}\sum_{\substack{Y\in\Pf(\Spec(\Z))\\ \sum_{p\in Y}r_K(p)=t}}\;\;\prod_{p\in Y}\beta_{p,n}.$$
\end{proposition}
\begin{proof}
Let $\X=\{X\in\Pf(\Spec(\OO_K))\mid |X|=t\}$ and $$\Y=\Big\{Y\in \Pf(\Spec(\Z))\mid \{\p\in \Spec(\OO_K)\mid \p\cap\Z\in Y\}\in \X\Big\}.$$
Therefore,
$\Y=\big\{Y\in \Pf(\Spec(\Z))\mid \sum_{p\in Y}r_K(p)=t\big\}$.
The result follows from Lemma~\ref{Lem X(K gamma) in fancy X}.
\end{proof}

Define
$$f_{K,n}(z)=\frac{\zeta(n+1)}{\zeta(n)}\prod_{p\in\Spec(\Z)} (1+z^{r_K(p)}\beta_{p,n}).$$
Recall that
$\beta_{p,n}=\frac{1}{p^{n}}\frac{1-\frac{1}{p}}{1-\frac{1}{p^n}}$. We immediately see the following.
\begin{lemma}
For any prime $p$ and $n\geq 2$, we have
$\frac{1}{2p^n} <\beta_{p,n}<\frac{1}{p^n}$.
\end{lemma}

\begin{lemma}
Given $n\geq 2$ and a number field $K$ with $[K:\Q]=d$, $f_{K,n}$ is an entire function of genus $\lfloor \frac{d}{n}\rfloor$.
\end{lemma}
\begin{proof}
First note that since $n\geq 2$ and $\beta_{p,n}<\frac{1}{p^n}$, the sum $\sum_{p}\beta_{p,n}$ converges. Next, for $R> 1$, consider some complex number with $|z|<R$. Then we have
$$\sum_{p}|z^{r_K(p)}\beta_{p,n}|< R^d\sum_{p}\beta_{p,n}<\infty.$$
This means the product that defines $f_{K,n}$ converges uniformly on compact sets, and hence $f_{K,n}$ is an entire function.

Next, note that the zeros of $f_{K,n}$ are $-\beta_{p,n}^{-\frac{1}{r_K(p)}} e^{\frac{2\pi i a}{r_K(p)}}$ for rational primes $p$ and integers $0\leq a<r_K(p)$. We first compute the rank of $f_{K,n}$. Let $\mu=\lfloor\frac{d}{n} \rfloor$, so $(\mu+1) \frac{n}{d}>1$ and $\mu\frac{n}{d}\leq 1$. Now
$$\sum_{p}\sum_{a=0}^{r_K(p)-1}\left|\beta_{p,n}^{\frac{1}{r_K(p)}}\right| ^{\mu+1}
\leq d\sum_{p} \left|\beta_{p,n}^{\frac{1}{r_K(p)}}\right|^{\mu+1}  
\leq d\sum_{p} \left|\beta_{p,n}^{\frac{1}{d}}\right|^{\mu+1} \leq d\sum_{p} \frac{1}{p^{(\mu+1)\frac{n}{d}}},$$
Which converges since $(\mu+1)\frac{n}{d}>1$. Next, notice that
$$\sum_{p}\sum_{a=0}^{r_K(p)-1}\left|\beta_{p,n}^{\frac{1}{r_K(p)}}\right| ^{\mu}
\geq \sum_{p} \left|\beta_{p,n}^{\frac{1}{r_K(p)}}\right|^{\mu}  
\geq \sum_{p: r_K(p)=d} \left|\beta_{p,n}^{\frac{1}{d}}\right|^{\mu}
\geq \frac{1}{2^{\frac{\mu}{d}}}\sum_{p: r_K(p)=d} \frac{1}{p^{\mu\frac{n}{d}}} 
\geq \frac{1}{2^{\frac{\mu}{d}}}\sum_{p: r_K(p)=d} \frac{1}{p}.$$
By the Chebotarev density theorem, we know that the rational primes that split completely in $K$ have a positive natural density. This implies (by Abel summation) that the sum $\sum_{p: r_K(p)=d} \frac{1}{p}$ diverges. Thus, the product that defines $f_{K,n}$ has rank $\mu$ and hence it has genus $\mu=\lfloor\frac{d}{n} \rfloor$.
\end{proof}

\begin{proposition}\label{Series expansion of fKn}
We have
$$f_{K,n}(z)=\sum_{t=0}^{\infty} a_{K,n,t} z^t,\;\;\;\;\;\E_n[|X(K,\gamma)|]=f_{K,n}'(1).$$
Moreover, for any complex number $c$,
$\E_n[e^{c|X(K,\gamma)|}]=f_{K,n}(e^c)$.
\end{proposition}
\begin{proof}
Notice that
\begin{equation*}
    \begin{split}
    f_{K,n}(z)
    &=\frac{\zeta(n+1)}{\zeta(n)}\prod_{p\in\Spec(\Z)} (1+z^{r_K(p)}\beta_{p,n})
    = \frac{\zeta(n+1)}{\zeta(n)} \sum_{Y\in \Pf(\Spec(\Z))} z^{\sum_{p\in Y}r_K(p)} \prod_{p\in Y}\beta_{p,n}\\
    &=\frac{\zeta(n+1)}{\zeta(n)}\sum_{t=0}^{\infty}z^t\sum_{\substack{Y\in\Pf(\Spec(\Z))\\ \sum_{p\in Y}r_K(p)=t}}\;\;\prod_{p\in Y}\beta_{p,n}
    = \sum_{t=0}^{\infty} z^t \PPP_n[|X(K,\gamma)|=t]\\
    &= \sum_{t=0}^{\infty} a_{K,n,t} z^t.
    \end{split}
\end{equation*}
Next, this implies that
$$f_{K,n}'(1) =\sum_{t=1}^{\infty} a_{K,n,t} t =\sum_{t=1}^{\infty}t \PPP_n[|X(K,\gamma)|=t] =\E[|X(K,\gamma)|].$$
Finally, note that
\[f_{K,n}(e^c) =\sum_{t=0}^{\infty} a_{K,n,t} e^{ct} =\sum_{t=0}^{\infty}e^{ct} \PPP_n[|X(K,\gamma)|=t] =\E\left[e^{c|X(K,\gamma)|}\right].\qedhere\]
\end{proof}

Let $m_{K,\gamma,s}$ be the $s^{th}$ moment of $|X(K,\gamma)|$.
Enumerate the rational primes as $p_1=2,p_2=3, p_3=5,\dots$.
\begin{corollary}\label{Cor: moments}
$$m_{K,\gamma,s}=\sum_{\substack{Y\subseteq\Spec(\Z)\\ Y\neq\emptyset,\; |Y|\leq s}}\prod_{p\in Y}\alpha_{p,n}\beta_{p,n}\sum_{A\subseteq Y}(-1)^{s-|A|}\left(\sum_{p\in Y}r_{K}(p)\right)^s$$
\end{corollary}
\begin{proof}
It is well known that the moment generating function is
$\sum_{s=0}^{\infty}\frac{m_{K,\gamma,s}}{s!}x^s=\E[e^{x|X(K,\gamma)|}]$.
Therefore, we see that
\begin{equation*}
    \begin{split}
    \sum_{s=0}^{\infty}\frac{m_{K,\gamma,s}}{s!}x^s
    &= f_{K,n}(e^{x})
    =\frac{\zeta(n+1)}{\zeta(n)}\prod_{p}(1+e^{xr_K(p)}\beta_{p,n})
    =\prod_{p}\alpha_{p,n} \prod_{p}\left(1+ \beta_{p,n}\sum_{t=0}^{\infty}\frac{r_{K}(p)^t}{t!}x^t \right)\\
    &= \prod_{p} \left(\alpha_{p,n}(1+\beta_{p,n}) +\alpha_{p,n}\beta_{p,n}\sum_{t=1}^{\infty}\frac{r_{K}(p)^t}{t!}x^t \right)
    =\prod_{p} \left(1 +\alpha_{p,n}\beta_{p,n}\sum_{t=1}^{\infty}\frac{r_{K}(p)^t}{t!}x^t \right)\\
    &=\sum_{s=0}^{\infty}x^s\sum_{u=1}^{s}\sum_{1\leq i_1<i_2<\dots<i_u}\prod_{j=1}^{u}\alpha_{p_{i_j},n}\beta_{p_{i_j},n}\sum_{\substack{t_1+\dots+t_u=s\\ t_j\geq 1}}\prod_{j=1}^u \frac{r_{K}(p_{i_j})^{t_j}}{t_j!}.
    \end{split} 
\end{equation*}
This implies that
\begin{equation*}
    \begin{split}
    m_{K,\gamma,s}
    &=\sum_{u=1}^{s}\sum_{1\leq i_1<i_2<\dots<i_u}\prod_{j=1}^{u}\alpha_{p_{i_j},n}\beta_{p_{i_j},n}\sum_{\substack{t_1+\dots+t_u=s\\ t_j\geq 1}} \frac{s!}{t_1!\dots t_u!} r_K(p_{i_1})^{t_1}\dots r_K(p_{i_u})^{t_u}\\
    &=\sum_{\substack{Y\subseteq\Spec(\Z)\\ Y\neq\emptyset,\; |Y|\leq s}}\prod_{p\in Y}\alpha_{p,n}\beta_{p,n}\sum_{A\subseteq Y}(-1)^{s-|A|}\left(\sum_{p\in Y}r_{K}(p)\right)^s.\qedhere
    \end{split}
\end{equation*}
\end{proof}

\begin{customprop}{\ref{Prop exp and var of size of X(K,gamma)}}
Given a number field $K$ and $n\geq 2$, we have
$$\E_n[|X(K,\gamma)|]=\sum_{p}\alpha_{p,n}\beta_{p,n}r_K(p),\;\;\;
\text{Var}_n[|X(K,\gamma)|]=\sum_{p}\alpha_{p,n}\beta_{p,n}(1-\alpha_{p,n}\beta_{p,n}) r_K(p)^2.$$
\end{customprop}
\begin{proof}
It follows from Corollary \ref{Cor: moments} that $\E_n[|X(K,\gamma)|]=m_{K,\gamma,1}=\sum_{p}\alpha_{p,n}\beta_{p,n}r_K(p)$ and
\begin{equation*}
    \begin{split}
    m_{K,\gamma,2}
    &=\sum_{p}\alpha_{p,n}\beta_{p,n}r_K(p)^2+ \sum_{p<q}\alpha_{p,n}\beta_{p,n}\alpha_{q,n}\beta_{q,n}\big((r_K(p)+r_K(q))^2-r_K(p)^2-r_K(q)^2\big)\\
    &=\sum_{p}\alpha_{p,n}\beta_{p,n}r_K(p)^2+ 2\sum_{p<q}\alpha_{p,n}\beta_{p,n}\alpha_{q,n}\beta_{q,n}r_K(p)r_K(q).
    \end{split}
\end{equation*}
Hence,
\[\text{Var}_n[|X(K,\gamma)|]=m_{K,\gamma,2}-m_{K,\gamma,1}^2 =\sum_{p}\alpha_{p,n}\beta_{p,n}(1-\alpha_{p,n}\beta_{p,n}) r_K(p)^2.\qedhere\]
\end{proof}

\begin{lemma}\label{Lemm: prod of t primes}
Given $\epsilon>0$, for sufficiently large $t$ we have
$$\prod_{i=1}^{t}p_i<e^{(1+\epsilon)t\log(t)}.$$
\end{lemma}
\begin{proof}
Let $\vartheta(X)=\sum_{p\leq X}\log(p)$. Therefore, $\prod_{1\leq i\leq t}p_i=e^{\vartheta(p_t)}$. Now, it is know that asymptotically $p_t\sim t\log(t)$ and $\vartheta(X)\sim X$. The result follows.
\end{proof}

\begin{customthm}{\ref{Thm bounds on prob size X(K, gamma) is t}}
Fix a number field $K$ and $n\geq 2$. Let $\mu=\lfloor\frac{d}{n}\rfloor$. Given $\epsilon>0$, for sufficiently large $t$, we have
$$t^{-(1+\epsilon)nt} < \PPP_n[|X(K,\gamma)|=t]< t^{-\frac{1}{\mu+1+\epsilon}t}.$$
\end{customthm}
\begin{proof}
Let $\lambda$ be the order of the entire function $f_{K,n}(z)$. Then we know that $\mu\leq \lambda \leq \mu+1$. Further from \cite[Theorem 2.2.2]{Boa}, we know that
$$\lambda =\limsup_{t\to\infty} \frac{t\log(t)}{\log(|a_{K,n,t}|^{-1})}.$$
This means that for large $t$, we have $\frac{t\log(t)}{\log(|a_{K,n,t}|^{-1})}< \lambda+\epsilon$, which means that
$$a_{K,n,t}< e^{-\frac{1}{\lambda+\epsilon}t\log(t)}.$$

From Proposition \ref{Prop prob size X(K,gamma) is t}, we know that
$$a_{K,n,t}=\frac{\zeta(n+1)}{\zeta(n)}\sum_{\substack{Y\in\Pf(\Spec(\Z))\\ \sum_{p\in Y}r_K(p)=t}} \prod_{p\in Y}\beta_{p,n}.$$
Define $c(t)$ to be the largest non-negative integer for which $\sum_{i=1}^{c}r_K(p_i)\geq t$.
Therefore we have $\sum_{i=1}^{c(t)-1}r_K(p_i)< t \leq\sum_{i=1}^{c(t)}r_K(p_i)$.
This implies that $c(t)-1< t\leq dc(t)$. Next denote
$\delta(t)= \sum_{i=1}^{c(t)}r_K(p_i)-t$.
Notice that $0\leq \delta(t)<r_K(p_{c(t)})\leq d$.

Now consider the $d$ smallest primes that split completely in $K$ (there are infinitely many such primes by the Chebotarev density theorem), suppose they are $p_{i_1},\dots,p_{i_d}$. Notice that $t\geq di_d$ implies that $c(t)\geq i_d$. Therefore for $t\geq di_d$, we have $\{p_{i_1},\dots,p_{i_d}\}\subseteq \{p_1,\dots, p_{c(t)}\}.$
Now consider
$Y_1=\{p_1,\dots, p_{c(t)}\}\setminus \{p_{i_1},\dots,p_{i_{\delta(t)}}\}$.
Therefore
$$\sum_{p\in Y_1}r_K(p)=\sum_{i=1}^{c(t)}r_K(p_i)-\sum_{j=1}^{\delta(t)}r_K(p_{i_j})= \sum_{i=1}^{c(t)}r_K(p_i)-\delta(t)=t.$$
By Lemma \ref{Lemm: prod of t primes}, this implies that for large $t$,
\begin{equation*}
    \begin{split}
        a_{K,n,t}
        &\geq \frac{\zeta(n+1)}{\zeta(n)} \prod_{p\in Y_1}\beta_{p,n}
        \geq \frac{\zeta(n+1)}{\zeta(n)} \prod_{i=1}^{c(t)}\beta_{p,n}
        \geq \frac{\zeta(n+1)}{\zeta(n)} \prod_{i=1}^{t}\beta_{p,n}\\
        &\geq \frac{\zeta(n+1)}{\zeta(n)} \prod_{i=1}^{t}\frac{1}{2p_i^{n}}
        \geq \frac{\zeta(n+1)}{\zeta(n)}\frac{1}{2^t} e^{-n(1+\frac{\epsilon}{2})t\log(t)} \geq e^{-n(1+\epsilon)t\log(t)}.\qedhere
    \end{split}
\end{equation*}
\end{proof}

Recall that we have
$$\PPP_n[|X(K,\gamma)|=t]=\frac{\zeta(n+1)}{\zeta(n)}\sum_{\substack{Y\in\Pf(\Spec(\Z))\\ \sum_{p\in Y}r_K(p)=t}}\;\;\prod_{p\in Y}\beta_{p,n}.$$
Let $d=[K:\Q]$.
Given $Y\in\Pf(\Spec(\Z))$ such that $\sum_{p\in Y}r_K(p)=t$, for $1\leq i\leq d$, construct $\lambda_i=\#\{p\in Y\mid r_K(p)=i\}$. Note that $\sum_{i=1}^{d}i\lambda_i=t$. If $K$ is Galois over $\Q$, then for those $i$ that do not divide $d$, we have $\lambda_i=0$.

\begin{definition}
Given a number field $K$ with $[K:\Q]=d$ and integers $1\leq j\leq d$, $t\geq 0$, $n\geq 2$.
We denote
$$\Lambda(d,t)=\Big\{\lambda=(\lambda_1,\dots,\lambda_d)\mid \lambda_j\geq 0, \sum_{j=1}^{d}j\lambda_j=t\Big\}.$$
Given $\lambda\in \Lambda(d,t_1)$ and $\lambda'\in \Lambda(d,t_2)$, we define $\lambda+\lambda'\in \Lambda(d,t_1+t_2)$ by coordinate wise addition.
Next, denote
$$\Lambda^0(d,t)=\Big\{\lambda\in \Lambda(d,t)\mid \lambda_j\neq 0 \text{ implies } j|d\Big\}.$$
Lastly, we denote
$$b_{K,n,j,t} =\sum_{\substack{1\leq i_1<i_2<\dots <i_{t}\\ r_K(p_{i_1})=\dots=r_K(p_{i_t})=j}}\beta_{p_{i_1},n}\dots \beta_{p_{i_t},n}.$$
\end{definition}

\begin{lemma}\label{Lemma: prob XKgamm =t in terms of lambda}
Given a number field $K$ with $[K:\Q]=d$, integers $n\geq 2$ and $t\geq 0$, we have
$$\PPP_n[|X(K,\gamma)|=t]=\frac{\zeta(n+1)}{\zeta(n)} \sum_{\lambda\in \Lambda(d,t)}\prod_{\substack{1\leq j\leq d,\\ \lambda_j\neq 0}} b_{K,n,j,\lambda_j}.$$
Moreover, if $K$ is Galois over $\Q$, then we have
$$\PPP_n[|X(K,\gamma)|=t]=\frac{\zeta(n+1)}{\zeta(n)} \sum_{\lambda\in \Lambda^0(d,t)}\prod_{\substack{j|d,\\ \lambda_j\neq 0}}b_{K,n,j,\lambda_j}.$$
\end{lemma}
\begin{proof}
The first part follows from Proposition \ref{Prop prob size X(K,gamma) is t}. For the second part, note that if $K$ is Galois over $\Q$ and $j\nmid d$, then $b_{K,n,j,t}=0$.
\end{proof}

Note that we do not always have $\PPP_n[|X(K,\gamma)|=t] >\PPP_n[|X(K,\gamma)|=t+1]$, for example $\PPP_n[|X(\Q(\sqrt{-7}),\gamma)|=1] <\PPP_n[|X(\Q(\sqrt{-7}),\gamma)|=2]$. However, we will show some weaker monotonicity related results for $\PPP_n[|X(K,\gamma)|=t]$.

\begin{lemma}\label{Lemma: prob xkg=a> xkg=a+b}
Suppose we are given positive integers $d,a,b$ such that the addition map from $\Lambda(d,a)\times \Lambda(d,b)$ to $\Lambda(d,a+b)$ is surjective. Then for every number field $K$ with $[K:\Q]=d$ and integer $n\geq 2$, we have $\PPP_n[|X(K,\gamma)|=a]>\PPP_n[|X(K,\gamma)|=a+b]$.
\end{lemma}
\begin{proof}
Firstly note that $b_{K,n,j,t_1}b_{K,n,j,t_2}\geq b_{K,n,j,t_1+t_2}$. Next, we apply Lemma \ref{Lemma: prob XKgamm =t in terms of lambda}.
\begin{equation*}
    \begin{split}
    \PPP_n[|X(K,\gamma)|=a+b]
    &=\frac{\zeta(n+1)}{\zeta(n)} \sum_{\lambda\in \Lambda(d,a+b)}\prod_{\substack{1\leq j\leq d,\\ \lambda_j\neq 0}} b_{K,n,j,\lambda_j}\\
    &\leq \frac{\zeta(n+1)}{\zeta(n)} \sum_{\lambda'\in \Lambda(d,a)}\sum_{\lambda''\in \Lambda(d,b)}\prod_{\substack{1\leq j\leq d,\\ \lambda'_j+\lambda''_j\neq 0}} b_{K,n,j,\lambda'_j+\lambda''_j}\\
    &\leq\frac{\zeta(n+1)}{\zeta(n)} \sum_{\lambda'\in \Lambda(d,a)}\prod_{\substack{1\leq j\leq d,\\ \lambda'_j\neq 0}} b_{K,n,j,\lambda'_j} \sum_{\lambda''\in \Lambda(d,b)}\prod_{\substack{1\leq j\leq d,\\ \lambda''_j\neq 0}} b_{K,n,j,\lambda''_j}\\
    &=\PPP_n[|X(K,\gamma)|=a] \Big(\frac{\zeta(n)}{\zeta(n+1)} \PPP_n[|X(K,\gamma)|=b]\Big).
    \end{split}
\end{equation*}    
So, we will be done if we show that $\frac{\zeta(n)}{\zeta(n+1)} \PPP_n[|X(K,\gamma)|=b]< 1$. We have $\PPP_n[|X(K,\gamma)|=0]=\frac{\zeta(n+1)}{\zeta(n)}$, so $\PPP_n[|X(K,\gamma)|=b]\leq 1-\frac{\zeta(n+1)}{\zeta(n)}$. Therefore, we only need to check $(\frac{\zeta(n+1)}{\zeta(n)})^{-1}(1-\frac{\zeta(n+1)}{\zeta(n)})<1$. This is equivalent to $\frac{\zeta(n+1)}{\zeta(n)}>\frac{1}{2}$, which is true for all $n\geq 2$.
\end{proof}

\begin{lemma}\label{Lemma: K galois decrease}
Suppose we are given positive integers $d,a,b$ such that the addition map from $\Lambda^0(d,a)\times \Lambda^0(d,b)$ to $\Lambda^0(d,a+b)$ is surjective. Then for every number field $K$ with $[K:\Q]=d$, which is Galois over $\Q$ and integer $n\geq 2$, we have $\PPP_n[|X(K,\gamma)|=a]>\PPP_n[|X(K,\gamma)|=a+b]$.
\end{lemma}
\begin{proof}
The proof is obtained by following a procedure identical to the proof of Lemma~\ref{Lemma: prob xkg=a> xkg=a+b}.
\end{proof}

\begin{proposition}\label{Prop: aKt are decreasing}
Suppose we are given a number field $K$ with $[K:\Q]=d$, $n\geq 2$. Let $d_1=lcm(1,2,\dots,d)$. For $t\geq (d-1)d_1-\frac{d(d+1)}{2}+1$, we have 
$$\PPP_n[|X(K,\gamma)|=t] >\PPP_n[|X(K,\gamma)|=t+d_1].$$
Let $\tau(d)$ be the number of divisors of $d$ and $\sigma(d)$ be the sum of the divisors of $d$. If $K$ is Galois over $\Q$, then for $t\geq d(\tau(d)-1)-\sigma(d)+1$, we have 
$$\PPP_n[|X(K,\gamma)|=t] >\PPP_n[|X(K,\gamma)|=t+d].$$
\end{proposition}
\begin{proof}
We will be done in the first part if we show that the map $\lambda+\lambda'$ from $\Lambda(d,t)\times \Lambda(d,d_1)$ to $\Lambda(d,t+d_1)$ is surjective.
Consider $\rho\in\Lambda(d,t+d_1)$. Note that
$$\sum_{j=1}^{d}j\rho_j=t+d_1\geq dd_1-\frac{d(d+1)}{2}+1>\sum_{j=1}^{d}(d_1-j).$$
This means there must be some $j$ for which $j\rho_j>d_1-j$, that is, $\rho_j>\frac{d_1}{j}-1$. Since $j\mid d_1$, we conclude that $\rho_j\geq \frac{d_1}{j}$. Consider $\lambda=(\rho_1,\dots,\rho_{j-1},\rho_j-\frac{d_1}{j},\rho_{j+1},\dots,\rho_d)$ and $\lambda'=(0,\dots,0,\frac{d_1}{j},0,\dots,0)$. It is clear that $\lambda\in \Lambda(d,t)$, $\lambda'\in \Lambda(d,d_1)$ and $\lambda+\lambda'=\rho$. Hence the map from $\Lambda(d,t)\times \Lambda(d,d_1)$ to $\Lambda(d,t+d_1)$ is surjective. We see that $\PPP_n[|X(K,\gamma)|=t] >\PPP_n[|X(K,\gamma)|=t+d_1].$

Next, suppose that $K$ is Galois over $\Q$ and $t\geq d(\tau(d)-1)-\sigma(d)+1$. We need to show that the addition map from $\Lambda^0(d,t)\times \Lambda^0(d,d)$ to $\Lambda^0(d,t+d)$ is surjective.
Consider $\rho\in\Lambda^0(d,t+d)$. Note that
$$\sum_{j|d} j\rho_j=t+d \geq d\tau(d)-\sigma(d)+1>\sum_{j|d}(d-j).$$
This means there must be some $j\mid d$ for which $j\rho_j>d-j$, that is, $\rho_j>\frac{d}{j}-1$. Since $j\mid d$, we conclude that $\rho_j\geq \frac{d}{j}$. Consider $\lambda\in\Lambda(d,t)$ given by $\lambda=(\rho_1,\dots,\rho_{j-1},\rho_j-\frac{d}{j},\rho_{j+1},\dots,\rho_d)$ and $\lambda'\in \Lambda(d,d)$ given by $\lambda'=(0,\dots,0,\frac{d}{j},0,\dots,0)$. It is clear that $\lambda+\lambda'=\rho$. Hence the map from $\Lambda^0(d,t)\times \Lambda^0(d,d)$ to $\Lambda^0(d,t+d)$ is surjective. We see that $\PPP_n[|X(K,\gamma)|=t] >\PPP_n[|X(K,\gamma)|=t+d].$
\end{proof}

\begin{corollary}
Given $n\geq 2$ and $t\geq 0$, we have
$$\PPP_n[|X(\Q,\gamma)|=t] >\PPP_n[|X(\Q,\gamma)|=t+1].$$
\end{corollary}

\begin{proposition}
Suppose $K$ is a number field which is Galois over $\Q$ and for which with $[K:\Q]=p^m$ is a prime power. Suppose we are given $n\geq 2$ and $t\geq 0$.  Then we have
$$\PPP_n[|X(K,\gamma)|=t] >\PPP_n[|X(K,\gamma)|=t+p^m].$$
Further consider $0\leq j\leq m-1$ and $t\geq 0$. Suppose $t$ satisfies the property that there is no $0\leq a\leq p^j-1$ for which $t\equiv (p-1)p^{j}+a\pmod{p^{j+1}}$. Then we have
$$\PPP_n[|X(K,\gamma)|=t] >\PPP_n[|X(K,\gamma)|=t+p^j].$$
\end{proposition}
\begin{proof}
For the first part, we want to          show that the addition map from $\Lambda^0(p^m,t)\times \Lambda^0(p^m,p^m)\to\Lambda^0(p^m,t+p^m)$ is surjective. Consider $\rho\in \Lambda^0(p^m,t+p^m)$, we know that $\sum_{i=0}^{m}p^i\rho_{p^i}=t+p^m$. Let $j_0=\min\{0\leq j\leq m\mid \sum_{i=j}^{m}p^i\rho_{p^i}\geq p^m\}$. This means that
$0\leq \sum_{i=j_0}^{m}p^i\rho_{p^i}-p^m< p^{j_0}\rho_{p^{j_0}}$. Let $a=\sum_{i=j_0}^{m}p^{i-j_0}\rho_{p^i}-p^{m-j_0}$, so $a\in \Z$ and $0\leq a<\rho_{p^{j_0}}$. Define $\lambda\in \Lambda^{0}(p^m,t)$ and $\lambda'\in\Lambda^0(p^m,p^m)$ as follows
\begin{align*}
\lambda_{p^i}&=
\begin{cases}
    0 &\text{if }i>j_0\\
    \rho_{p^{j_0}}-a &\text{if } i=j_0\\
    \rho_{p^i} &\text{if } i<j_0,
\end{cases}
&
\lambda'_{p^i}&=
\begin{cases}
    \rho_{p^i} &\text{if }i>j_0\\
    a &\text{if } i=j_0\\
    0 &\text{if } i<j_0.
\end{cases}
\end{align*}
It is clear that $\lambda+\lambda'=\rho$. Thus the addition map from $\Lambda^0(p^m,t)\times \Lambda^0(p^m,p^m)\to\Lambda^0(p^m,t+p^m)$ is surjective.

For the second part suppose we are given $0\leq j\leq m-1$ and $t\geq 0$, such that there is no $0\leq a\leq p^j-1$ for which $t+p^j\equiv a\pmod{p^{j+1}}$. We want to show that the addition map from $\Lambda^0(p^m,t)\times \Lambda^0(p^m,p^j)\to\Lambda^0(p^m,t+p^j)$ is surjective.
Consider $\rho\in \Lambda^0(p^m,t+p^j)$, we know that $\sum_{i=0}^{m}p^i\rho_{p^i}=t+p^j$. We see that $t+p^j\equiv \sum_{i=0}^{j}p^i\rho_{p^i} \pmod{p^{j+1}}$, so $\sum_{i=0}^{j}p^i\rho_{p^i}\geq p^j$. 
Let $k_0=\min\{0\leq k\leq j\mid \sum_{i=k}^{j}p^i\rho_{p^i}\geq p^j\}$. This means that
$0\leq \sum_{i=k_0}^{j}p^i\rho_{p^i}-p^j< p^{k_0}\rho_{p^{k_0}}$. Let $b=\sum_{i=k_0}^{j}p^{i-k_0}\rho_{p^i}-p^{j-k_0}$, so $b\in \Z$ and $0\leq b<\rho_{p^{k_0}}$. Define $\lambda\in \Lambda^{0}(p^m,t)$ and $\lambda'\in\Lambda^0(p^m,p^j)$ as follows
\begin{align*}
\lambda_{p^i}&=
\begin{cases}
    \rho_{p^i} &\text{if }j<i\\
    0 &\text{if }k_0<i\leq j\\
    \rho_{p^{j_0}}-b &\text{if } i=k_0\\
    \rho_{p^i} &\text{if } i<k_0,
\end{cases}
&
\lambda'_{p^i}&=
\begin{cases}
    0 &\text{if }j<i\\
    \rho_{p^i} &\text{if }k_0<i\leq j\\
    b &\text{if } i=k_0\\
    0 &\text{if } i<k_0.
\end{cases}
\end{align*}
It is clear that $\lambda+\lambda'=\rho$. Thus the addition map from $\Lambda^0(p^m,t)\times \Lambda^0(p^m,p^m)\to\Lambda^0(p^m,t+p^m)$ is surjective.
\end{proof}

\section{Subgroup of Class group generated by $X(K,\gamma)$.}\label{Sec: group gen by XKgamma}

In this section we will study the group $\big\langle [\p],\p\in X(K,\gamma)\big\rangle$, particularly when $K$ is a quadratic imaginary field. We will prove Theorem \ref{Thm: <X K gamma>=G}.

\begin{lemma}
Suppose we are given a number field $K$ and integers $t\geq 1$ and $n\geq 2$. Let $Y'$ be the set of rational primes $p$ for which at least one prime $\p$ of $\OO_K$ above $p$ has $\p^t$ not principal. Then we have
$$\PPP_n\big[ \big\langle [\p],\p\in X(K,\gamma)\big\rangle \text{ is }t\text{-torsion}\big]
=\prod_{p\in Y'} \alpha_{p,n}.$$
\end{lemma}
\begin{proof}
Firstly $\big\langle [\p],\p\in X(K,\gamma)\big\rangle$ is $t$-torsion if and only if for each prime $\p\in X(K,\gamma)$, $\p^t$ is principal. Let $X$ be the collection of primes $\p$ of $\OO_K$ for which $\p^t$ is principal. Let $Y$ be the collection of primes $p$ of $\Z$, such that all primes of $\OO_K$ above $p$ are in $X$. Then by Proposition \ref{Class group OK[gamma]cap K}, we know that
\[\PPP_n\big[ \big\langle [\p],\p\in X(K,\gamma)\big\rangle \text{ is }t\text{-torsion}\big] =\PPP_n[X(K,\gamma)\subseteq X] =\prod_{p\notin Y}\alpha_{p,n}.\qedhere\]
\end{proof}



\begin{lemma}
For $n\geq 2$ and $N\geq 2$, we have
$$1-\frac{1}{(n-1)(N-1)^{n-1}}<\prod_{p\geq N}\alpha_{p,n}<1 <\prod_{p\geq N}\alpha_{p,n}^{-1}<1+\frac{1}{(n-1)(N-1)^{n-1}}.$$
\end{lemma}
\begin{proof}
Firstly, for all primes $p$ we have $\alpha_{p,n}<1$. Therefore, $\prod_{p\geq N}\alpha_{p,n}<1$.
Note that $\beta_{p,n}=\frac{1}{p^n}\frac{1-\frac{1}{p}}{1-\frac{1}{p^{n+1}}}<\frac{1}{p^n}$. This implies that
\begin{equation*}
    \begin{split}
    \alpha_{p,n}^{-1}
    &=\prod_{p\geq N}(1+\beta_{p,n})
    <\prod_{p\geq N}(1+\frac{1}{p^n})< 1+\sum_{m=N}^{\infty}\frac{1}{m^n}\\ &<1+\int_{N}^{\infty} \frac{1}{(x-1)^n}dx =1+\frac{1}{(n-1)(N-1)^{n-1}}.
    \end{split}
\end{equation*}
From this we see that
\[\prod_{p\geq N}\alpha_{p,n}>\frac{1}{1+\frac{1}{(n-1)(N-1)^{n-1}}}>1-\frac{1}{(n-1)(N-1)^{n-1}}.\qedhere\]
\end{proof}

Consider the character $\chi_2$,
$$\chi_2(a)=\begin{cases} 0 &\text{ if } 2\mid a\\ 1 &\text{ if } a\equiv \pm 1\pmod{8}\\ -1 &\text{ if } a\equiv \pm 5\pmod{8}.\end{cases}$$
For odd primes $p$, consider characters $\chi_p(a)=\legendre{a}{p}$.
For each $n\geq 2$, define functions
\begin{align*}
f_n(a)&=\prod_{p: \chi_p(a)\in\{0,1\}}\alpha_{p,n},
&
g_n(a)&=\prod_{p: \chi_p(a)\in\{1\}}\alpha_{p,n}.
\end{align*}

\begin{proposition}
For each quadratic imaginary field $K$, denote its discriminant as $d_K$.
Fix integers $n\geq 2$ and odd integer $t\geq 1$. Then we have
$$\lim_{d_K\to-\infty}\frac{\PPP_n\big[ \big\langle [\p],\p\in X(K,\gamma)\big\rangle \text{ is }t\text{-torsion}\big]}{f_n(d_K)}=1.$$
\end{proposition}
\begin{proof}
Fix $N>0$. Consider a quadratic imaginary field $K$ with discriminant $d_K<-4N^t$ and a prime $p<N$. Let $\p_1$ be a prime of $\OO_K$ above $p$. We will show that $\p_1^t$ is not principal if and only if $\chi_p(d_K)\in \{0,1\}$.
\begin{itemize}
    \item Case 1: $\chi_p(d_K)=-1$. In this case $p$ remains inert in $\OO_K$, so the $\p_1=(p)$ and $\p_1^t=(p^t)$ is principal.
    \item Case 2: $\chi_p(d_K)\in\{0,1\}$. In this case $p$ either splits or ramifies in $\OO_K$. Say $p\OO_K=\p_1\p_2$, so $N(\p_1)=p$ and $N(\p_1^t)=p^t$ (it is possible that $\p_1=\p_2$). Now if $\p_1^t$ is principal then there is an element in $\OO_K$ of norm $p^t$. Since $\OO_K$ is a subset of the $\Z$ module generated by $\frac{1}{2}, \frac{\sqrt{d_K}}{2}$, this would imply that there is an integer solution to
    $\frac{a^2}{4}-d_k\frac{b^2}{4}=p^t$.
    However $-\frac{d_k}{4}>N^t>p^t$, so $b=0$. But then we have integer solution to $a^2=4p^t$, which is impossible since $t$ is odd. Therefore, neither of $\p_1^t$ or $\p_2^t$ can be principal. 
\end{itemize}
Let $Y'_K$ denote the set of rational primes $p$ for which at least one prime $\p$ of $\OO_K$ above $p$ has $\p^t$ not principal.
We see that for $d_K<-4N^t$, we have
$$\frac{\PPP_n\big[\big\langle [\p],\p\in X(K,\gamma)\big\rangle \text{ is }t\text{-torsion}\big]}{f_n(d_K)} =\prod_{p\in Y'_K, \;p\geq N}\alpha_{p,n} \prod_{p\geq N, \chi_p(d_K)\in\{0,1\}}\alpha_{p,n}^{-1}.$$
Finally notice that
$$\prod_{p\in Y_K'\;p\geq N}\alpha_{p,n} \prod_{p\geq N, \chi_p(d_K)\in\{0,1\}}\alpha_{p,n}^{-1}
\leq \prod_{p\geq N}\alpha_{p,n}^{-1} \leq 1+\frac{1}{(n-1)(N-1)^{n-1}},$$
and
$$\prod_{p\in Y_K',\;p\geq N}\alpha_{p,n} \prod_{ p\geq N, \chi_p(d_K)\in\{0,1\}}\alpha_{p,n}^{-1}
\geq \prod_{p\geq N}\alpha_{p,n} \geq 1-\frac{1}{(n-1)(N-1)^{n-1}}.$$
Since $N$ was arbitrary, we are done.
\end{proof}

\begin{corollary}\label{cor size odd >=3}
For any odd $t\geq 3$, we have
$$\lim_{d_K\to-\infty} \PPP_n\big[ \#\big\langle [\p],\p\in X(K,\gamma)\big\rangle=t\big]=0.$$
\end{corollary}
\begin{proof}
Firstly notice that
\begin{equation*}
    \begin{split}
     &\PPP_n\big[ \#\big\langle [\p],\p\in X(K,\gamma)\big\rangle=t\big]\\
    &\leq \PPP_n\big[ \big\langle [\p],\p\in X(K,\gamma)\big\rangle \text{ is }t\text{-torsion}\big]
    -\PPP_n\big[ \big\langle [\p],\p\in X(K,\gamma)\big\rangle \text{ is }1\text{-torsion}\big].   
    \end{split}
\end{equation*}
Therefore,
$$\limsup_{d_K\to-\infty} \frac{\PPP_n\big[ \#\big\langle [\p],\p\in X(K,\gamma)\big\rangle=t\big]}{f_n(d_K)}\leq 1-1=0,$$
so
$$\lim_{d_K\to-\infty} \frac{\PPP_n\big[ \#\big\langle [\p],\p\in X(K,\gamma)\big\rangle=t\big]}{f_n(d_K)}=0.$$
Next note that for any $d_K$, we have
$$f_n(d_K)=\prod_{p: \chi_p(d_K)\in\{0,1\}}\alpha_p \geq \prod_{p}\alpha_p =\frac{\zeta(n+1)}{\zeta(n)}.$$
Therefore, we can conclude that
\[\lim_{d_K\to-\infty} \PPP_n\big[ \#\big\langle [\p],\p\in X(K,\gamma)\big\rangle=t\big]=0.\qedhere\]
\end{proof}

\begin{proposition}
For each quadratic imaginary field $K$, denote its discriminant as $d_K$.
Fix integers $n\geq 2$ and even integer $t\geq 2$. Then we have
$$\lim_{d_K\to-\infty}\frac{\PPP_n\big[ \big\langle [\p],\p\in X(K,\gamma)\big\rangle \text{ is }t\text{-torsion}\big]}{g_n(d_K)}=1.$$
\end{proposition}
\begin{proof}
Say $t=2t_1$, fix $N>0$. Consider a quadratic imaginary field $K$ with discriminant $d_K<-4N^t$ and a prime $p<N$. Let $\p_1$ be a prime of $\OO_K$ above $p$. We will show that $\p_1^t$ is not principal if and only if $\chi_p(d_K)=1$.
\begin{itemize}
    \item Case 1: $\chi_p(d_K)=-1$. In this case $p$ remains inert in $\OO_K$, so the $\p_1=(p)$ and $\p_1^t=(p^t)$ is principal.
    \item Case 2: $\chi_p(d_K)=0$. This means that $p$ ramifies so $p\OO_K=\p_1^2$. Therefore $\p_1^t=p^{t_1}\OO_K$.
    \item Case 3: $\chi_p(d_K)=1$. In this case $p$ splits in $\OO_K$. Say $p\OO_K=\p_1\p_2$, so $N(\p_1)=p$ and $N(\p_1^t)=p^t$ and $\p_1\neq\p_2$. Now if $\p_1^t$ is principal, say $\p_1^t=\alpha\OO_K$, then $N(\alpha)=p^t$. Since $\OO_K$ is a subset of the $\Z$ module generated by $\frac{1}{2},\frac{\sqrt{d_K}}{2}$, let $\alpha=\frac{a+b\sqrt{d_k}}{2}$. Then
    $\frac{a^2}{4}-d_k\frac{b^2}{4}=p^t$.
    However $-\frac{d_k}{4}>N^t>p^t$, so $b=0$. Thus $a=\pm 2p^{t_1}$ and $\alpha=\pm p^{t_1}$. This means that $\p_1^t=\alpha\OO_K= p^{t_1}\OO_K$. But $p^{t_1}\OO_K=\p_1^{t_1}\p_2^{t_1}$. Thus we see that $\p_1=\p_2$, which is a contradiction.
\end{itemize}
Let $Y'_K$ denote the set of rational primes $p$ for which at least one prime $\p$ of $\OO_K$ above $p$ has $\p^t$ not principal.
We therefore see that for $d_K<-4N^t$, we have
$$\frac{\PPP_n\big[ \big\langle [\p],\p\in X(K,\gamma)\big\rangle \text{ is }t\text{-torsion}\big]}{g_n(d_K)} =\prod_{p\in Y_K',\;p\geq N}\alpha_{p,n} \prod_{p\geq N, \chi_p(d_K)=1}\alpha_{p,n}^{-1}.$$
Finally, notice that
$$\prod_{p\in Y_K',\;p\geq N}\alpha_{p,n} \prod_{ p\geq N, \chi_p(d_K)=1}\alpha_{p,n}^{-1}
\leq \prod_{\; p\geq N}\alpha_{p,n}^{-1} \leq 1+\frac{1}{(n-1)(N-1)^{n-1}},$$
and
$$\prod_{p\in Y_K',\;p\geq N}\alpha_{p,n} \prod_{p\geq N, \chi_p(d_K)=1}\alpha_{p,n}^{-1}
\geq \prod_{p\geq N}\alpha_{p,n} \geq 1-\frac{1}{(n-1)(N-1)^{n-1}}.$$
Since $N$ was arbitrary, we are done.
\end{proof}

\begin{customthm}{\ref{Thm: <X K gamma>=G}}
Suppose $G$ is a finite abelian group which is not 2-torsion (in particular, $G$ is not trivial). Fix $n\geq 2$. For each quadratic imaginary field $K$, denote its discriminant as $d_K$. Then we have
$$\lim_{d_K\to-\infty}\PPP_n\big[ \big\langle [\p],\p\in X(K,\gamma)\big\rangle \cong G\big]=0.$$    
\end{customthm}
\begin{proof}
Let $|G|=m$. If $m$ is odd, then we are done by Corollary \ref{cor size odd >=3}. Therefore suppose that $m$ is even. Since $G$ is not $2$-torsion, we see that
\begin{equation*}
    \begin{split}
&\PPP_n\big[ \big\langle [\p],\p\in X(K,\gamma)\big\rangle \cong G\big]\\
&\leq \PPP_n\big[ \big\langle [\p],\p\in X(K,\gamma)\big\rangle \text{ is }m\text{-torsion}\big]
-\PPP_n\big[ \big\langle [\p],\p\in X(K,\gamma)\big\rangle \text{ is }2\text{-torsion}\big].
    \end{split}
\end{equation*}
Therefore,
$$\limsup_{d_K\to-\infty} \frac{\PPP_n\big[ \big\langle [\p],\p\in X(K,\gamma)\big\rangle \cong G\big]}{g_n(d_K)}\leq 1-1=0,$$
and hence
$$\lim_{d_K\to-\infty} \frac{\PPP_n\big[ \big\langle [\p],\p\in X(K,\gamma)\big\rangle\cong G\big]}{g_n(d_K)}=0.$$
Next, note that for any $d_K$, we have
$$g_n(d_K)=\prod_{p:\; \chi_p(d_K)=1}\alpha_{p,n} \geq \prod_{p}\alpha_{p,n} =\frac{\zeta(n+1)}{\zeta(n)}.$$
Therefore, we can conclude that
\[\lim_{d_K\to-\infty} \PPP_n\big[ \big\langle [\p],\p\in X(K,\gamma)\big\rangle\cong G\big]=0.\qedhere\]
\end{proof}

\section{Distribution of $r_K(p)$}\label{Sec: rkp}
In this section we consider the splitting of a prime $p$ in $\OO_{\Q(\alpha)}$ as we sample $\alpha$ from algebraic integers $\A'_m(H)$. Our goal is to prove Theorem \ref{Thm: prob rkp=i}.
Denote 
$$\FF_m(H)=\{f\in \Z[x]\mid \deg(f)=m,H(f)\leq H, f\text{ is monic}\},$$
and
$$\LL_m'(H)=\{f(x)\in \FF_m(H)\mid f\text{ is irreducible in }\Z[x]\}.$$
It is clear that $\#\A_m'(H)=m\#\LL'_m(H)$.
By Lemma \ref{Hilbert Irr bound}, we know that $\#\LL'_m(H)=(2H+1)^{m}+O(H^{m-\frac{1}{2}}\log(H))$. Further, note that $\#\FF_m(H)=(2H+1)^{m}$. This means that $\FF_m(H)-\LL'_m(H)=O(H^{m-\frac{1}{2}}\log(H))$ and $\lim_{H\to\infty}\frac{\LL'_m(H)}{\FF_m(H)}=1$.

\begin{lemma}\label{Lemma: irr to all}
Given a positive integer $m$ and a subset $A\subseteq \bigcup_{H=1}^{\infty}\FF_m(H)$. Then
$$\lim_{H\to\infty} \frac{\#(A\cap \LL'_m(H))}{\#\LL'_m(H)}= \lim_{H\to\infty}\frac{\#(A\cap \FF_m(H))}{\#\FF_m(H)}.$$
\end{lemma}
\begin{proof}
Suppose $\lim_{H\to\infty}\frac{\#(A\cap \FF_m(H))}{\#\FF_m(H)}=a$.
Then we have
$$\limsup_{H\to\infty} \frac{\#(A\cap \LL'_m(H))}{\#\LL'_m(H)}
\leq \limsup_{H\to\infty} \frac{\#(A\cap \FF_m(H))}{\#\LL'_m(H)}
=\lim_{H\to\infty}\frac{\#(A\cap \FF_m(H))}{\#\FF_m(H)}=a.$$
Further, notice that
\begin{equation*}
\begin{split}
\liminf_{H\to\infty} \frac{\#(A\cap \LL'_m(H))}{\#\LL'_m(H)}
&\geq \liminf_{H\to\infty} \frac{\#(A\cap \FF_m(H))-\#(\FF_m(H)\setminus \LL'_m(H)) }{\#\FF_m(H)}\\
&=\lim_{H\to\infty}\frac{\#(A\cap \FF_m(H))}{\#\FF_m(H)}-\lim_{H\to\infty} \frac{O(H^{m-\frac{1}{2}}\log(H))}{(2H+1)^{m}}=a.\qedhere
\end{split}
\end{equation*}
\end{proof}

Denote
$$\FF_{m,p}=\{f(x)\in (\Z/p\Z)[x]\mid \deg(f)=m, f\text{ is monic}\}.$$

\begin{lemma}
Given $m\in\Z_{>0}$ and a prime $p\geq m$, we have $\#\{f\in \FF_{m,p}\mid  \disc(f)=0\}\leq p^{m-1}m$. For $\alpha\in \A'_m(H)$, if $p\mid [\OO_{\Q(\alpha)}:\Z[\alpha]]$ then $p\mid \disc(\alpha)$. Moreover, we have
$$\limsup_{H\to\infty}\frac{\#\{\alpha\in\A'_m(H)\mid p\mid \disc(\alpha)\}}{\#\A'_m(H)}\leq \frac{m}{p}.$$
\end{lemma}
\begin{proof}
Writing $f(x)=x^m+a_{1}x^{m-1}+\dots+a_m$, we know that $\disc(f)$ is a polynomial in $a_{1},a_{2},\dots,a_{m}$. Write it as $\sum_{i=0}^{m}f_i(a_1,\dots,a_{m-2},a_m)a_{m-1}^{i}$. We know from Lemma~\ref{An-1^n} that $f_{m}(a_1,\dots,a_{m-2},a_m)=\pm (m-1)^{m-1}$. Once we choose $a_1,\dots,a_{m-2},a_{m}\in \Z/p\Z$, we have at most $m$ choices for $a_{m-1}$. Therefore,
$\#\{f\in \FF_{m,p}\mid  \disc(f)=0\}\leq p^{m-1}m$.
    
We know that for any $\alpha \in \Zbar$, we have $\Z[\alpha] \subseteq \OO_{\Q(\alpha)} \subseteq \frac{1}{\disc(\alpha)}\Z[\alpha]$. Consequently, $[\OO_{\Q(\alpha)}:\Z[\alpha]] \mid \disc(\alpha)^m$. Therefore, if $p \mid [\OO_{\Q(\alpha)}:\Z[\alpha]]$, then $p \mid \disc(\alpha)$. Finally, notice that
\begin{equation*}
    \begin{split}
\frac{\#\{\alpha\in\A'_m(H)\mid p\mid \disc(\alpha)\}}{\#\A'_m(H)}
&=\frac{\#\{f\in \LL'_m(H)\mid p\mid \disc(f)\}}{\#\LL'_m(H)}\\
&\leq \frac{\#\{f\in \FF_m(H)\mid p\mid \disc(f)\}}{(2H+1)^m+O(H^{m-\frac{1}{2}}\log(H))}\\
&\leq \frac{\#\{f\in \FF_{m,p}\mid  \disc(f)=0\}(\frac{2H+1}{p}+1)^{m}}{(2H+1)^m+O(H^{m-\frac{1}{2}}\log(H))}\\
&\leq \frac{p^{m-1}m(\frac{2H+1}{p}+1)^{m}}{(2H+1)^m+O(H^{m-\frac{1}{2}}\log(H))}\\
&=\frac{1}{p}\frac{m}{\big(\frac{2H+1}{2H+1+p}\big)^m+O\Big(\big(\frac{H}{2H+1+p}\big)^m  \frac{\log(H)}{\sqrt{H}}\Big)}.
    \end{split}
\end{equation*}
Therefore,
\[\limsup_{H\to\infty}\frac{\#\{\alpha\in\A'_m(H)\mid p\mid \disc(\alpha)\}}{\#\A'_m(H)} \leq \frac{m}{p}.\qedhere\]

\end{proof}

Given $1\leq i\leq m$ and prime $p$, denote
$$f(m,i,p)=\frac{\#\{f\in \FF_{m,p}\mid f \text{ has exactly }i\text{ distict irreducible factors}\}}{\#\FF_{m,p}},$$
and $f(m,i)=\lim_{p\to\infty}f(m,i,p)$.

\begin{lemma}
Given $1\leq i\leq m$ and a prime $p\geq m$, we have $g(m,i,p)=f(m,i,p)+O(\frac{1}{p})$. And hence $g(m,i)=f(m,i)$.
\end{lemma}
\begin{proof}
Suppose we have $\alpha\in \A'_m(H)$ such that $p\nmid \disc(\alpha)$. Then we know that $p\nmid [\OO_{\Q(\alpha)}:\Z[\alpha]]$. Hence, by Dedekind Kummer theorem we know that $r_{\Q(\alpha)}(p)=i$ if and only if the minimal polynomial of $\alpha$ factors into $i$ distinct factors mod $p$. Therefore, we see that
\begin{align*}
g(m,i,p)&=\lim_{H\to\infty}\frac{\#\{\alpha\in\A'_m(H)\mid r_{\Q(\alpha)}(p)=i\}}{\#\A'_m(H)}\\
&=\lim_{H\to\infty}\frac{\#\{\alpha\in\A'_m(H)\mid r_{\Q(\alpha)}(p)=i,p\nmid \disc(\alpha)\}}{\#\A'_m(H)}+\Op\\
&=\lim_{H\to\infty}\frac{\#\{f\in\LL'_m(H)\mid f\text{ mod $p$ has }i\text{ distinct factors} ,p\nmid \disc(f)\}}{\#\LL'_m(H)}+\Op\\
&=\lim_{H\to\infty}\frac{\#\{f\in\FF_m(H)\mid f\text{ mod $p$ has }i\text{ distinct factors} ,p\nmid \disc(f)\}}{\#\FF_m(H)}+\Op\\
&=\lim_{H\to\infty}\frac{f(m,i,p)p^{m} (\frac{2H+1}{p}+O(1))^{m}}{(2H+1)^m}+\Op=f(m,i,p)+\Op.\qedhere
\end{align*}
\end{proof}

Now our task boils down to computing $f(m,i)$.
Let $a_m(p)$ denote the number of irreducible monic polynomials of degree $m$ in $(\Z/p\Z)[x]$. It is well known that
$$a_{m}(p)=\frac{1}{m}\sum_{d \mid m}\mu(\frac{m}{d})p^d=\frac{1}{m}p^m+O(p^{\frac{m}{2}}).$$
If $\lambda$ is a partition, denote the number of times $i$ occurs in $\lambda$ as $b_i(\lambda)$. Let $\PP(m,i)$ denote the set of all partitions of $m$ into $i$ parts.

\begin{lemma}
Given $1\leq i\leq m$, we have
$$f(m,i)=\sum_{\lambda\in \PP(m,i)}\prod_{n=1}^{m}\frac{1}{n^{b_n(\lambda)}b_n(\lambda)!}.$$
\end{lemma}
\begin{proof}
First we notice that $f$ has repeated roots if and only if $\disc(f)=0$. Next, for each $f\in \FF_{m,p}$ with $\disc(f)\neq 0$, associate the partition $\lambda$ such that $b_n(\lambda)$ is the number of irreducible factors of $f$ of degree $n$. Now, notice that
\begin{align*}
f(m,i,p)&=\frac{\#\{f\in \FF_{m,p}\mid f \text{ has exactly }i\text{ distict irreducible factors}\}}{\#\FF_{m,p}}\\
&=\frac{\#\{f\in\FF_{m,p}\mid \disc(f)=0\}}{\#\{f\in \FF_{m,p}\}}
+ \sum_{\lambda\in\PP(m,i)}\frac{\#\{f\in \FF_{m,p}\mid f\text{ associated to }\lambda\}}{\#\{f\in \FF_{m,p}\}}\\
&=O\Big(\frac{1}{p}\Big)+\sum_{\lambda\in\PP(m,i)}\frac{1}{p^m}\prod_{n=1}^{m}\binom{a_n(p)}{b_n(\lambda)}\\
\end{align*}
\begin{align*}
&=O\Big(\frac{1}{p}\Big)+\sum_{\lambda\in\PP(m,i)}\frac{1}{p^m}\prod_{n=1}^{m}\frac{(a_n(p)+O(1))^{b_n(\lambda)}}{b_n(\lambda)!}\\
&=O\Big(\frac{1}{p}\Big) +\sum_{\lambda\in\PP(m,i)}\frac{1}{p^m}\prod_{n=1}^{m}\frac{(\frac{1}{n}p^n+O(p^{\frac{n}{2}}))^{b_n(\lambda)}}{b_n(\lambda)!}\\
&=\sum_{\lambda\in\PP(m,i)}\frac{1}{p^m}\prod_{n=1}^{m}\frac{\frac{1}{n^{b_n(\lambda)}}p^{nb_n(\lambda)}+O(p^{nb_n(\lambda)-\frac{n}{2}})}{b_n(\lambda)!}+O\Big(\frac{1}{p}\Big)\\
&=\sum_{\lambda\in\PP(m,i)}\prod_{n=1}^{m}\frac{1}{n^{b_n(\lambda)}b_n(\lambda)!} +\Op.\qedhere
\end{align*}
\end{proof}

\begin{customthm}{\ref{Thm: prob rkp=i}}
Given $m$ and $i$ such that $1\leq i\leq m$, $g(m,i)$ is the coefficient of $y^i$ in $\frac{\prod_{j=0}^{m-1}(y+j)}{m!}$.    
\end{customthm}
\begin{proof}
We know that $g(m,k)=\sum_{\lambda\in \PP(m,k)}\prod_{n=1}^{m}\frac{1}{n^{b_n(\lambda)}b_n(\lambda)!}$. Now consider the power series
\begin{equation*}
\begin{split}
\sum_{m=0}^{\infty}\sum_{k=1}^mx^my^kg(m,k) 
&= \sum_{m=0}^{\infty}\sum_{k=1}^mx^my^k \sum_{\lambda\in \PP(m,k)}\prod_{n=1}^{m}\frac{1}{n^{b_n(\lambda)}b_n(\lambda)!} 
=\prod_{n=1}^{\infty}\Big(1+\sum_{l=1}^\infty \frac{x^{nl}y^l}{l!n^l} \Big) \\
&=\prod_{n=1}^\infty \exp\Big({\frac{x^ny}{n}}\Big) 
=\exp\Big(y\sum_{n=1}^\infty \frac{x^n}{n}\Big)
=\exp{\Big(-y\log(1-x)\Big)} \\
&=(1-x)^{-y} 
=\sum_{m=0}^{\infty}\binom{-y}{m}(-1)^mx^m 
=\sum_{m=0}^\infty \frac{\prod_{j=0}^{m-1}(y+j)}{m!}x^m
\end{split}
\end{equation*}
Therefore, $g(m,k)$ is the coefficient of $y^k$ in $\frac{\prod_{j=0}^{m-1}(y+j)}{m!}$.
\end{proof}

\end{document}